\pdfoutput = 1
\documentclass{article}
\usepackage{verbatim}

\usepackage{amsmath}
\usepackage{amsthm}
\usepackage{amssymb}
\usepackage{latexsym}
\usepackage{staves}
\usepackage{marvosym}
\usepackage{eurosym}
\usepackage{phaistos}
\usepackage{verbatim}

\usepackage{color}
\usepackage[T1]{fontenc}
\usepackage[english]{babel}
\usepackage[utf8]{inputenc}
\usepackage{graphicx}
\usepackage{subfigure}
\usepackage{tikz}
\usetikzlibrary{arrows}
\usetikzlibrary{arrows,decorations.markings}
\usetikzlibrary{mindmap,trees}
\usepackage{hyperref}

\newtheorem{Theorem}{Theorem}[part]
\newtheorem{Definition}{Definition}[part]
\newtheorem{Proposition}{Proposition}[part]

\newtheorem{Assumption}{Assumption}[part]

\newtheorem{Remark}{Remark}[part]

\makeatletter \@addtoreset{equation}{section}

\@addtoreset{Definition}{section}

\@addtoreset{Theorem}{section}

\@addtoreset{Proposition}{section}

\@addtoreset{Property}{section}

\@addtoreset{Assumption}{section}

\@addtoreset{Corollary}{section}

\@addtoreset{Lemma}{section}

\@addtoreset{Remark}{section}

\@addtoreset{Example}{section}

\def\Ac{{\cal A}}

\def\eps{\varepsilon}

\def\eps{\varepsilon}

\def\no{\noindent}

\def\05{\frac{1}{2}}
\def\-1{^{-1}}
\def\1{{1\hspace{-1mm}{\rm I}}}
\def\={\;=\;}
\def\.{\;.}

\title{  }
\author{ }

\def\be{\begin{eqnarray}}
\def\ee{\end{eqnarray}}

\def\b*{\begin{eqnarray*}}
\def\e*{\end{eqnarray*}}

\def\Ac{{\cal A}}

\def\-1{^{-1}}

\def\0.5{\frac{1}{2}}

\def\no{\noindent}
\def\={\;=\;}
\def\.{\;.}

\def\eps{\varepsilon}

\def\1{{\bf 1}}

\title{MFG model with a long-lived penalty at random jump times: application to demand side management \\ for electricity contracts}
\author{Cl\'emence Alasseur \thanks{EDF R\&D and Finance for Energy Market Research Centre (FIME), France, email: \texttt{clemence.alasseur@edf.fr}} \and Luciano Campi\thanks{Department of Mathematics "Federico Enriques", University of Milan, Italy, email: \texttt{luciano.campi@unimi.it}} \and Roxana Dumitrescu\thanks{Department of Mathematics, King's College London, United Kingdom, email: \texttt{roxana.dumitrescu@kcl.ac.uk}} \and Jia Zeng\thanks{Department of Mathematics, King's College London, United Kingdom and 
The University of Hong Kong, Hong Kong, email: \texttt{jia.zeng@kcl.ac.uk}}}

\begin{document}

\maketitle

\begin{abstract}
	We consider an energy system with $n$ consumers who are linked by a Demand Side Management (DSM) contract, i.e.  they agreed to diminish, at  random times, their aggregated power consumption by a predefined volume during a predefined duration. Their failure to deliver the service is penalised via the difference between the sum of the $n$ power consumptions and the contracted target. We are led to analyse a non-zero sum stochastic game with $n$ players, where the interaction takes place through a cost which involves a delay induced by the duration included in the DSM contract. When $n \to \infty$, we obtain a Mean-Field Game (MFG) with random jump time penalty and interaction on the control. We prove a stochastic maximum principle in this context, which allows to compare the MFG solution to the optimal strategy of a central planner. In a linear quadratic setting we obtain an semi-explicit solution through a system of decoupled forward-backward stochastic differential equations  with jumps, involving a Riccati Backward SDE with jumps. We show that it provides an approximate Nash equilibrium for the original $n$-player game for $n$ large. Finally, we propose a numerical algorithm to compute the MFG equilibrium and present several numerical experiments.\medskip
	
	\no {\bf Keywords:} Demand Side Management, Real-Time Pricing, mean-field games, mean-field control, delay, Riccati BSDE with jumps, stochastic maximum principle.
\end{abstract}

\section{Introduction}

Dynamic pricing or real time pricing (RTP) for electricity contracts is a form of demand side management (DSM) and is possible due to the large development of smart meters. Through dynamic pricing, the consumer can optimise her flexibility over time as she is aware of potential scarcity or on the contrary oversupply on the production side. This kind of contracts is indeed a way to enhance active participation of customers to contribute to the balancing of  consumption-production equilibrium which becomes more challenging with the development of intermittent renewable  and the shut down of some thermal production plants. This type of contracts is bound to develop as in Europe, the Clean Energy Package\footnote{https://ec.europa.eu/energy/topics/energy-strategy/clean-energy-all-europeans\_en}, which is central to define European roadmap to decarbonize its energy sector, aims at each final customer being entitled to choose a dynamic electricity price contract by its supplier. In 2019, seven states in European  Union\footnote{see Energy prices and costs in Europe, European Commission, 2019} (Estonia, Finland, Sweden, Spain, Netherlands, Denmark and the United Kingdom) were proposing those type of contracts. Other forms of DSM contracts include interruptible load contracts which are activated by the Transmission System Operator (TSO) to contribute to the balancing of the power system and to maintain grid and system security. Interruptible load  are often contracted by large consumers with industrial process but also by aggregator who represents a large collection of small consumers.

\vspace{0.5cm}
\textbf{Proposed Mean-Field Game model of flexible customers with dynamic pricing}. The aim of our paper is to provide a stylized  model to analyse a system with DSM contracts involving a large fraction of clients. Our model includes both RTP and interruptible load contracts. RTP implies that power consumers are exposed to a variable price, e.g. the spot or the real time price. Interruptible load contracts are often managed by an aggregator which by definition aggregates several consumers. By this contract, the aggregator is committed to reduce the global consumption of its clients' portfolio by a certain amount and for a given duration at specific  moments corresponding to tension on global equilibrium. The aggregator is of course penalized if he does not manage to achieve the reduction of consumption he was committed to perform, that is if the sum of the reactions of all customers he aggregates does not fit the contracted target during the whole duration of the contract. The triggered instants of the interruptible load contract are often decided by the Transmission System Operator (TSO) which is the operator in charge of the balance of the system. Our model does not represent the production side of the power system and the interruptible load contract is considered to be triggered randomly in our model through a jump process. The aggregator is not represented explicitly in our model: the DSM is operated by customers themselves who already agreed with the aggregator on the structure of the contract. The power system is then modeled as a large population of consumers who manage individually their consumption and provides DSM services to minimize the cost of their retail contract. In our model, each consumer is characterised by a two-state variable, their natural power demand $Q_t$ and the accumulated deviation $S_t$ from natural power demand, and a control variable given by the deviation $\alpha_t$ from natural power demand. We propose general dynamics for natural power demand with Brownian and jump components to be able to fit observed natural demand of individual consumers. The objective of each consumer is to minimize its own cost of electricity and inconvenience cost to deviate from natural power demand. Since we consider that a large fraction of consumers are active and involved in this type of DSM contracts, we assume that the real-time price and the respect of the engagement of the interruptible load when activated is impacted by the total deviation of all customers involved in the DSM. Therefore, we are led to the analysis of a non-zero sum stochastic game with $n$ players in a non-cooperative game setting and to the search of Nash equilibria. As $n$ is typically very large, we rely on a Mean Field Game (MFG) approach with \textit{common noise}, \textit{random jump times penalty} and \textit{interaction on the control}. 

%\vspace{0.5cm}
\paragraph{Literature review.}
MFG have been introduced simultaneously in \cite{lasry2006jeux-I,lasry2006jeux-II} and \cite{huang2006} as natural limits of symmetric stochastic differential games for a large population of players interacting through a mean field. Since then, a very rich literature has been developed from both perspectives - theory and applications. For a complete treatment of the probabilistic theory of MFG we refer to \cite{CarmonaDelarueBook17}, while for the applications to economics and finance we cite the recent survey \cite{carmona2020applications}. 
This model is an extension to \cite{Alasseur2020} with introduction of jumps in the state variable dynamics and a long lived penalty at random jump times in the cost function, which, in the particular case of a quadratic cost structure and linear  pricing and divergence rules, leads to a linear-quadratic model with jumps and random coefficients. For the particular case of the model studied in \cite{Alasseur2020}, we provide a more detailed analysis from both theoretical and numerical points of view (see the paragraph below for details).
Related works in the case of a Brownian filtration can be found in e.g. \cite{Graber2016}, \cite{yong13}, \cite{yong99}, \cite{sun2014}, \cite{pham2016}, while in the case of jumps we cite the articles \cite{benazzoli2019varepsilon,benazzoli2020mean} for jump-diffusion state variable dynamics and \cite{li2019mean} for a pure-jump MFG model of cryptocurrency mining. Several other MFG models have been developed to study the management of consumers' flexibility with interaction through spot prices represented as inverse functions of the total power demand - see \cite{Couillet21012} for EV (Electrical Vehicle) charging,  \cite{Paola2016} for micro-storages or  \cite{Paola2019} for TCL (Thermostatic Control Load). In \cite{Paola2019}, the production system is  explicitly represented and agents also interact through frequency response. In \cite{Bauso2017}, the authors study a demand side management problem for TCL devices, by considering an MFG formulation which involves interaction on the distribution of their temperature. In \cite{Gomes2020}, power consumers interact through price which is not an inverse demand function, but the result of the equilibrium of the power system.

%\vspace{0.5cm}
\paragraph{Main contributions.} Our model provides an analytically and numerically tractable setting to assess questions
related to development and practical implementation of DSM on power system. The associated game is formulated as an MFG with \textit{common noise}, \textit{interaction on the control} and \textit{random jump time penalty} and we provide some conditions under which there exists an unique equilibrium.   We are able to show that the MFG is equivalent to a Mean-Field Type Control (MFC) with suitable pricing and divergence rules. This connection enables to decentralize the aggregator's optimisation problem to the customer's level which is much more tractable in practice, avoiding to rely on heavy communication system. In the particular case where the cost structure is quadratic and the pricing and divergence rules are linear, the mean-field equilibrium is characterised through a decoupled system of Forward Backward Stochastic Differential
Equations with jumps, involving a Riccati BSDE with jumps. We propose a numerical scheme and provide several numerical illustrations together with detailed interpretations of the results. Finally, in the linear-quadratic case, we prove rigorously that the MFG solution is  an $\varepsilon$-Nash equilibrium for the $n$-player game. To the best of our knowledge, very few MFG models with jumps have been studied in the previous literature (see paragraph above) and  this is the first paper which proposes a linear quadratic  model with \textit{common noise}, \textit{jumps} and \textit{random coefficients}, for which complete mathematical and numerical treatments are provided.

\paragraph{Organization of the paper.} The paper is organized as follows: in Section 2 we describe the n-players model. In Section 3, we formulate the mean-field game and the mean-field optimal control problems and characterize the optimal solutions, by providing, under specific conditions, the relation between the two. In Section 4, we study the linear-quadratic setting, in particular we characterize the optimal control through a system of forward backward SDEs with jumps and show that the equilibrium strategy provides an approximated Nash equilibria for the $n$-players game. In Section 6, we describe the numerical approach to compute the mean-field equilibria and provide numerical results.
 %\medskip

\section{The Model}

We consider a large number, say $n$, of electricity consumers who have entered a \textit{Demand Side Management} (DSM) contract. This means that they agreed to modify their electricity consumption by postponing, anticipating or even giving up some energy at some inconvenience cost.

\begin{figure}[!ht]
    \centering
    \includegraphics[scale=0.5]{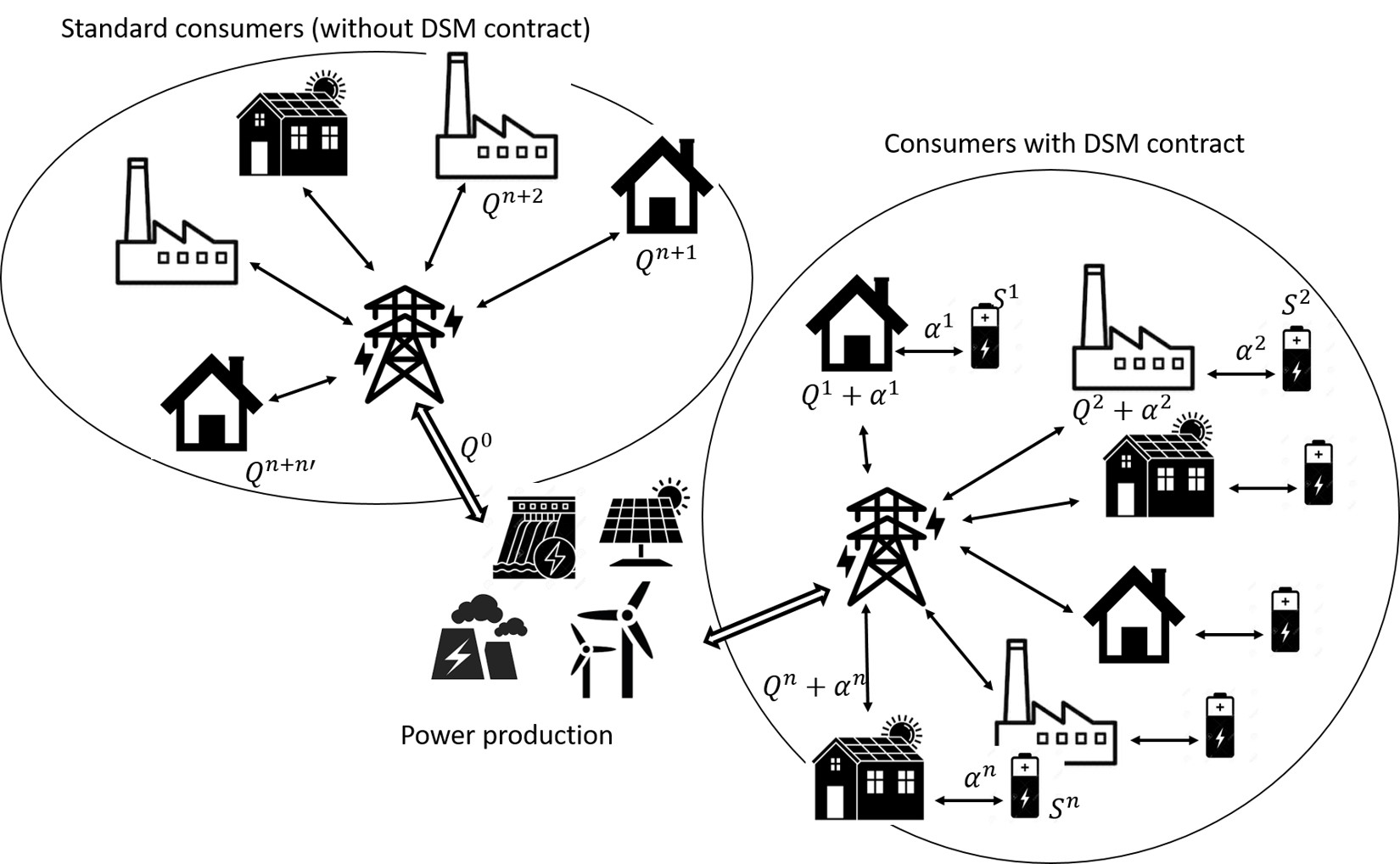}
    \caption{Illustration of the model}
    \label{fig:global_scheme}
\end{figure}
Each consumer involved in the DSM contract $i =1,\ldots,n$ is characterised by two state variables $(Q^i ,S^i)$. The first state variable $Q_t^i$ is the instantaneous electricity  consumption of consumer $i$ which represents the volume of electricity  consumer $i$ needs at time $t$. The consumer is also active: at time $t$ the consumer can decide to deviate by an amount $\alpha^i _t$ from his natural power demand and to finally consume $(Q^i _t + \alpha^i _t )dt$. When $\alpha^i _t>0$ (resp. $<0$), the consumer consumes more (resp. less) that he would naturally need. The accumulated volume of consumption the consumer has agreed to modify from his natural power demand $Q^i _t$ since the beginning of the time period is denoted by $S^i _t$.  Typically, $S$ could represent the  state of charge of a battery but it could also represent the cumulated effort the consumer makes by postponing or anticipating some actions like washing, charging his electrical vehicle (EV) and so forth. \medskip

Another category of consumers is also present in the market that is the \textit{standard consumers} who do not optimise their consumption. Each standard consumer $i = n+1 ,\ldots,n' $ is characterised by one state variable only $Q^i$. The average consumption of the standard consumers is $Q^{{st}} = \frac{1}{n'}\sum_{i=n+1}^{n'}Q^{i}$. \medskip

We consider  a complete probability space $(\Omega, \mathcal{F}, \mathbb{P})$ on which are defined independent Brownian motions $W^0, W^1, \ldots, W^n$, independent Poisson processes $ N^1, \ldots, N^n$ with intensity given by a positive real number $\lambda$ and a counting process $N^0$ with intensity $\lambda^0$ and independent of the individual noises. The process $\widetilde N_t^0 $ (resp.   $\widetilde N_t^{i}$, $i=1,...,n$) represents the compensated martingale and is given by $\widetilde N_t^0 = N_t^0 - \lambda^0 t$ (resp. $\widetilde N_t^{i}= N_t^{i}-\lambda t$, $i=1,...,n$). We  consider $n$ independent identically distributed (i.i.d.)  random variables $x_0^i=(q^i_0, s^i_0)$ which are independent of $W^0$, $W^i$, $N^i$ and $N^0$. We denote by $\mathbb F=(\mathcal{F}_t)_{t\in [0,T]}$ the filtration generated by $(s^i_0, q^i_0), N^0, W^0, N^i, W^i, i=1, \ldots ,n$, and satisfying the usual conditions. We denote by $\mathcal{A}$ the set of $\mathbb{F}$-adapted real-valued processes $\alpha=(\alpha_t)_{t \in [0,T]}$ such that
$ \mathbb{E}\left[ \int_0^T |\alpha_t|^2 dt \right]< \infty$ and $\mathbb{E}[|\alpha_\tau|\mathbf 1_{\tau<\infty}]<\infty$ for all $\mathbb F^0$-stopping times $\tau$ with values in $[0,T]\cup \{+\infty\}$. The latter property grants existence of the optional projection for any $\alpha \in \mathcal A$, that we will need later on in the paper (see Remark \ref{rmk-opt}). \medskip

The dynamics of the two state variables for consumer $i$ are given by: 
\be \nonumber
dQ_t^i &=& \mu(Q^i_t, t)dt + \sigma (Q^i_t, t)dW^i_t+\beta (Q^i_{t^-}, t)dN^i_t + \sigma^{0} (Q^i_t, t)dW^0_t ,\quad Q_0^i = q_0^i ,\\ 
dS_t^i &=& \alpha^i_t dt, \quad S^i _0 = s_0^i.\nonumber
\ee
The dynamic of the average standard consumption $Q^{st}$ is given by: 
\be \nonumber
dQ_t^{st} &=& \mu^{st}(Q^{st}_t, t)dt + \sigma^{st} (Q^{st}_t, t)dW^0_t+\beta^{st} (Q^{st}_t, t)d\Tilde{N}^0_t ,\quad Q_{0}^{st} = q_0^{st} ,\nonumber
\ee
\textcolor{black}{where the coefficients $\mu(t,x)$, $\mu^{st}(t,x)$, $\sigma(t,x)$, $\sigma^{st}(t,x)$, $\beta^{st}(t,x)$, $\beta^{st}(t,x)$ are continuous in $(t,x)$ and Lipschitz continuous with respect to $x$, uniformly in $t$, which ensures the existence of unique strong solutions for the above SDEs.} \\
Furthermore, we denote by $\tilde Q_t^i$ the deseasonalised version of the state variable $Q_t^i$ such that $\tilde Q_t^i = Q_t^i -\mathbb{E}\left[Q_t^i\right]$. $\tilde Q_t^i$ corresponds to the divergence of the consumption to the seasonal consumption ($\mathbb{E}\left[Q_t^i\right]$). \medskip

In the model, the consumer is not restricted to be consumer-only; he may have his own local production (for example he could own solar panels). Therefore, consumer $i$ could be, at time $t$, either in a \textit{consuming mode} meaning he needs electricity ($Q_t^i$ positive) or in a \textit{producing mode} ($Q_t^i$ negative) when he is a net producer.\medskip

Our model represents an interruptible load contract which is activated by the TSO when tension happens on the  production-consumption balance. The aim of this contract is to equalize the total power deviation $\sum_i \alpha_t^i$ to a deterministic contracted target $\bar\alpha$ at random instants $(\tau_k)$ decided by the TSO.   
$(\tau_k)_{k \ge 1}$ is an increasing sequence of stopping times which are the jump times of the counting process $N^0$.  In addition, the interruptible load contract is such that the effort $\bar\alpha$ is maintained during a deterministic duration $\theta$. To avoid trivial cases, let us assume $\theta < T$, with $T$ the finite time horizon.
During interruptible load contract activation, each agent $i$ is penalised when the total response $\sum_i \alpha_t^i$ differs from the requirement $\bar\alpha$. The energy operator cannot monitor $\alpha$ but only $Q+\alpha$ at each consumer because  power meters register the global consumption only. With those measures, the aggregator is able to estimate the deseasonalised consumption $\tilde Q + \alpha$ which corresponds to the divergence of the consumption to the standard consumption (i.e. $\mathbb{E}\left[Q_t^i\right]$). 
The divergence cost is therefore proportional to this quantity $\tilde Q+ \alpha$. \medskip

This divergence cost is the following:
\[ d_t^i =  J_t^\theta  (\tilde Q_t^i + \alpha_t^i - \bar \alpha) f \left(\frac{1}{n}\sum_{j=1}^n  (\tilde  Q_t^j + \alpha_t^j ) - \bar \alpha \right)\]
with $f$ a convex growing function such as $f(0)=0$ and $J_t^\theta$ equal to one during interruptible load contract activation and 0 otherwise.\medskip

To specify $J^\theta$, we introduce a new state variable, $R_t$, which measures the time since the last DSM jump occurred. Once a jump occurs, the process is set to $0$. The process $R_t$ is given by the unique strong solution of the following SDE:
\begin{align*}
dR_t=dt-R_{t^-}dN^0_t,\,\, R_0=2\theta,
\end{align*}
 and $J_t^{\theta}  = \textbf{1}_{R_t \leq \theta}$ in $d_t$. \medskip

In addition, power is not free and the consumers also have to pay their power at a RTP. We consider in this model that consumer are charged at the spot price of the power system  which is an inverse demand function $p$ of the expectation of consumption of all involved consumers.  The \emph{power cost} $c_t$ is therefore, at time $t$, given by: 
\[ c_t^i = (Q_t^i+\alpha_t^i )p\left(\underbrace{\frac{1}{n+n'}\sum_{j=n+1}^{n'}Q^{j}}_{\text{standard consumers}}+\underbrace{\frac{1}{n+n'}\sum_{j=1}^n (Q_t^j+ \alpha_t^j)}_{\text{consumers with DSM contract}}   \right).\] 
The proportion of standard consumers with respect to DSM consumers in the total population is $\pi = \frac{n'}{n+n'}$ and the \emph{power cost} $c_t^i$ can be rewritten as: 
\[ c_t^i = (Q_t^i+\alpha_t^i )p\left(\pi Q^{st}_t + (1-\pi)\frac{1}{n} \sum_{j=1}^n (Q_t^j+ \alpha_t^j)  \right),\]
\[ c_t^i = (Q_t^i+\alpha_t^i )p_t\left(\frac{1}{n}\sum_{j=1}^n (Q_t^j+ \alpha_t^j)  \right),\] with 
\begin{align}\label{price_fct}
p_t(Q) = p\left(\pi Q^{st}_t + (1-\pi)Q\right).
\end{align}
\medskip

To capture the costs induced by the efforts made by the agent when controlling its consumption, we include an \emph{inconvenience cost} $g(\alpha_t, S_t, Q_t)$. Typically, $g$ is convex and increasing in both $\alpha_t$ and $S_t$. This represents that the more the consumer needs to deviate from his natural consumption $Q_t$, the more disrupting. In addition, large accumulated deviations $S_t$ are also penalized. We also consider an additional cost $l(Q_t,\alpha_t)$ to represent demand charge component of the retail tariff structure. Most electricity bills are indeed structured in two parts: first the energy consumption, the amount of energy (kWh) consumed, multiplied by the relevant price of energy and secondly the demand charge: the maximum amount of power (kW) drawn for any given time interval (typically 15 minutes) during the billing period, multiplied by the relevant demand charge. 
 \medskip

Finally, consumers are facing \emph{a terminal cost} $h(S_T)$. Indeed, $S_T \neq 0$ means that the agent did not get during the period $[0,T]$ the exact amount of energy he needed. Therefore, the terminal cost penalises this extra or negative amount of energy the consumer has to manage during the period. \medskip

Finally, each consumer $i \in \{1,\ldots, n\}$ wants to minimise its total expected costs:
\be\nonumber
 \inf_{\alpha^i \in \mathcal A} J^i_n(\alpha) = \inf_{\alpha^i \in \mathcal A} \mathbb{E}\left[\int_{0}^{T} \left( g(\alpha^i_t, S_t^i, Q_t^i)  +l(Q_t^i+\alpha_t^i) + c^i_t + d_t^i \right)dt  + h(S_T^i)\right],
\ee
with $\alpha=(\alpha^{1}, \ldots, \alpha^{n}) $. \medskip

We are led to the analysis of a non-zero sum stochastic game with $n$ players and  to the search of  Nash equilibria:
\begin{Definition}[$\eps$-Nash equilibrium for the $n$-players game]
Let $\eps \ge 0$. We say that a strategy profile $\alpha^\star=(\alpha^{\star,1}, \ldots, \alpha^{\star,n}) \in \Ac^n$ is a  $\eps$-Nash-equilibrium if for each $i$, for any $\beta \in \Ac$:~
	$$
	J^{i}_n(\alpha^{\star,1}, \ldots, \alpha^{\star,i-1},\beta, \alpha^{\star,i+1} , \ldots,\alpha^{\star,n})
	\ge
	J^{i}_n(\alpha^{\star,1}, \ldots,  \alpha^{\star, n}) - \eps
	.$$
\end{Definition}

\section{The mean-field game and mean-field type control problems}

In this part, we consider the mean-field game (MFG) formulation arising at the limit when $n \rightarrow \infty$, which essentially refers to considering the optimization problem of a \textit{representative consumer} and looking for the existence an \textit{equilibrium}. We also consider a different, but still related optimization problem of \textit{mean-field type control} (MFC), which consists in assigning a strategy to all agents at once, such that the resulting crowd behavior is optimal with respect to costs imposed on a \textit{central planner}. Let us first describe the mathematical framework.

Fix a terminal time $T>0$. Let $(\Omega, \mathcal{F}, \mathbb{P})$ be a probability space, equipped with two independent Brownian motions $W^{0}$ and $W$ and two independent Poisson processes $N^0$ and $N$ (also independent of $W$ and $W^0$) with intensities two positive real numbers $\lambda^0$ and $\lambda$. We denote by $\Tilde{N}^0$ (resp.$\Tilde{N}$) the compensated Poisson processes, i.e. $\Tilde{N}^0_t:=N^0_t-\lambda^0t$ (resp. $\Tilde{N}_t:=N_t-\lambda t$). Let $(s_0, q_0)$ be two random variables independent of $W, W^0, N, N^0$. Let $\mathbb{F}= (\mathcal{F}_t)_{t\in [0,T]}$ be the (complete) natural filtration generated by ($W$, $W^0$, $N$, $N^0$, $s_0$, $q_0$). Let $\mathbb{F}^0=(\mathcal{F}^0_t)_{t\in [0,T]}$ be the (complete) natural filtration generated by ($W^0$, $N^0$). Furthermore, the state variables $(Q,S)$ have dynamics
\be
dQ_t &=& \mu (Q_t, t)dt + \sigma (Q_t, t)dW_t + \beta(Q_{t^-}, t)d\Tilde{N}_t+ \sigma^{0} (Q_t, t)dW^0_t ,\quad Q_0= q_0 , \label{dynQ}\\ 
dS_t &=& \alpha_t dt, \quad S _0 = s_0. \label{dynS}
\ee
We denote by $\tilde Q_t = Q_t -\mathbb{E}\left[Q_t\right]$, $t \in [0,T]$.

\begin{Remark} \label{rmk-opt}\normalfont
{\color{black}Given any $\mathcal B([0,T]) \otimes \mathcal F$-measurable process $X$ such that $\mathbb E[|X_\tau | \mathbf 1_{\tau<\infty}]<\infty$ for all $\mathbb F^0$-stopping times $\tau$ with values in $[0,T]\cup \{+\infty\}$, the notation $\widehat{X}$ indicates the optional projection of $X$ with respect to the filtration $\mathbb F^0$, i.e. $\widehat X$ is the unique (up to indistinguishibility) $\mathbb F^0$-optional process such that $\widehat X_\tau 1_{\tau<\infty} = \mathbb E[X_\tau 1_{\tau<\infty} | \mathcal F_\tau ^{0}]$ a.s. for all $\mathbb F^0$-stopping times $\tau$ with values in $[0,T]\cup \{+\infty\}$ (cf. Section 2 in \cite{bremaud1978changes}). We will use this notation throughout the whole paper.}
\end{Remark}

We now move on to the next sub-section, where we give the detailed formulation of the MFG and MFC problems.

\subsection{Formulation of the representative consumer and central planner problems}

We start with the mean-field game problem.
\paragraph{The representative consumer: MFG problem.}
Let $\xi = (\xi_t)_{t\in [0,T]}$ be a given $\mathbb{F}^{0}$-adapted process.
Consider the objective functional
\begin{align}
    J^{MFG}(\alpha; \xi)= &\mathbb{E}\left[\int_{0}^{T} \left( g(\alpha_t, S_t, Q_t)+l(Q_t+\alpha_t)  +(Q_t+\alpha_t )p_t\left( \widehat{Q}_t + {\xi}_t\right)
  \right. \right. \nonumber \\ 
&\left. \left. + \, J_t^\theta (\tilde Q_t + \alpha_t - \bar \alpha) f \left(\widehat{\tilde Q_t} +  \xi_t   - \bar \alpha \right) \right)dt  + h(S_T)\right],
\end{align}
where $\alpha = (\alpha_t)_{t\in [0,T]}$ is an \textit{admissible} control process which belongs to $\mathcal A$, the set of all real-valued $\mathbb F$-adapted processes such that $\mathbb E[\int_0 ^T \alpha_t ^2 dt] < \infty$ and $\mathbb{E}[|\alpha_\tau|\mathbf 1_{\tau<\infty}]<\infty$ for all $\mathbb F^0$-stopping times $\tau$ with values in $[0,T]\cup \{+\infty\}$. The latter requirement guarantees that for any $\alpha \in \mathcal A$ the optional projection $\widehat \alpha$ with respect to $\mathbb F^0$ is well-defined (see Remark \ref{rmk-opt}) and satisfies $\widehat \alpha_\tau 1_{\tau<\infty} = \mathbb E[\alpha_\tau 1_{\tau<\infty} | \mathcal F_\tau ^{0}]$ a.s. for all $\mathbb F^0$-stopping times $\tau$ with values in $[0,T]\cup \{+\infty\}$. Notice that the a-priori estimates of the solution of the equation \eqref{dynQ} imply that the optional projections $\widehat Q$ and $\widehat{\tilde Q}$, appearing in the objective functional above, are also well defined. 

The optimization problem of the representative consumer can be written as follows
\be\nonumber
 V^{MFG}(\xi) =  \inf_{\alpha \in \mathcal A} J^{MFG}(\alpha; \xi).
\label{eq_2.1}
\ee
The goal is to find a process $\alpha^\star = (\alpha_t^\star)_{t\in [0,T]}$ such that
\begin{align}
    J^{MFG}(\alpha^\star; \xi)= V^{MFG}(\xi) \quad\,\,\,\,\text{and}\quad\quad\,\,\,\,{\color{black}\widehat \alpha_t^\star =\xi_t,\,\,\text{a.s. for all }t \in [0,T].}
\end{align}
Such a process $\alpha^\star$ is called a \textit{mean-field Nash equilibrium}.\\

We make the following assumption on the coefficients, which ensures that the problem is well-defined.

\begin{Assumption}
\begin{enumerate}
    \item $g: \mathbb{R}^3 \to \mathbb{R}$, $l: \mathbb{R} \to \mathbb{R}$ and $h:\mathbb{R} \to \mathbb{R}$ have at most quadratic growth and are strictly convex.
    \item $p: \mathbb{R} \to \mathbb{R}, f: \mathbb{R} \to \mathbb{R}$ have at most linear growth.
    \item $g$, $p$, $f$, $l$ and $h$ are differentiable.
\end{enumerate}
\label{assump_2.1}
\end{Assumption}

\vspace{2mm}
We now formulate the mean-field type control problem.
\paragraph{The central planner: MFC problem.}
The mean-field type control problem corresponds to the problem of a \textit{central planner} who wants to optimise the global population (standard and DSM consumers). Standard consumers are only charged at the spot price and cost $l$. The objective functional takes the following form
\begin{align}
J^{C}(\alpha)= &\mathbb{E} \left[(1-\pi)\int_{0}^{T} \left( g(\alpha_t, S_t, Q_t)  + (Q_t+\alpha_t )p_t\left(\widehat{Q}_t + \widehat{\alpha}_t \right)  \right.\right. \nonumber\\ 
& \hspace{1cm} \left.\left. + l(Q_t+\alpha_t)+ J_t^\theta (\tilde Q_t + \alpha_t - \bar \alpha) f \left(\widehat{\tilde Q_t} + \widehat \alpha_t   - \bar \alpha \right) \right) dt +(1-\pi)h(S_T)  \right.\nonumber\\ 
&  \hspace{1cm}  \left.\pi \int_{0}^{T} \left(  Q_t^{st} p_t\left(\widehat{Q}_t + \widehat{\alpha}_t \right)  + l(Q_t^{st}) \right) dt \right].
\end{align}
The optimization problem of the central planner writes as follows:
\begin{align}
V^C = &\inf_{\alpha \in \mathcal{A}} J^{C}(\alpha).
\end{align}

\subsection{Characterization of the MFG equilibria and optimal MFC}

\paragraph{Characterization of the MFG equilibria.}
We provide here a caracterization result of the MFG Nash equilibria. To this purpose, we first define the set $\mathcal{S}^2$ of $\mathbb{F}$-adapted RCLL processes $Y$ such that $E[\sup_{0 \leq t \leq T} Y_t^2]<\infty$ and the set $\mathcal{H}^2$ of $\mathbb{F}$-predictable processes $Z$ such that $\mathbb{E}[\int_0^T Z_s^2ds]<\infty$.

\begin{Theorem}
Let $\hat{\xi}$ be a given $\mathbb{F}^0$-adapted $\mathbb{R}$-valued process and $x_0=(s_0, q_0)$ be a random vector independent of $\mathbb{F}^0$. If there exists a control $\alpha^\star \in \mathcal{A}$ which minimizes the map $\alpha \mapsto J^{MFG}(\alpha, \hat{\xi})$  and if $(S^{\alpha^\star}, Q)$ is the state process associated to the initial condition $x_0$, control $\alpha^\star$ and the dynamics \eqref{dynQ}-\eqref{dynS}, then there exists a unique solution $(Y^\star,q^{0, \star},q^{\star}, \nu^\star, \nu^{0,\star}) \in \mathcal{S}^2 \times (\mathcal{H}^2)^4$ of the following BSDE with jumps:
\begin{align}
    -dY^\star_t=\partial_x g(\alpha, S_t ^{\alpha^\star} , Q_t)dt-q_t^{0, \star}dW_t^0-q^\star_t dW_t - \nu^\star_t d\widetilde N_t- \nu^{0,\star}_t d\widetilde N^0_t, \nonumber \\
    Y^\star_T=\partial_x h(S_T^{\alpha^\star}),\label{adj 1}
\end{align}
satisfying the coupling condition
\begin{align}\label{coupling1}
\partial_\alpha g(\alpha^{\star}_t, S^{\alpha^{\star}}_t, Q_t) +\partial_\alpha l(Q_t+ \alpha^{\star}_t)+{p_t}\left(\widehat{Q}_t+ \hat{\xi}_t\right)+Y^{\star} _t+ J_t^\theta  f \left(\widehat{\tilde Q_t} +  \hat{\xi}_t   - \bar \alpha \right) =0.
\end{align}
Conversely, assume that there exists $\left(\alpha^\star,S^{\alpha^\star}, Y^\star, q^{0,\star},q^{\star}, \nu^{\star}, \nu^{0,\star}\right) \in \mathcal A \times (\mathcal{S}^2)^2 \times (\mathcal{H}^2)^4$ satisfying the coupling condition $\eqref{coupling1}$, as well as the FBSDE $\eqref{dynS}-\eqref{adj 1}$, then $\alpha^\star$ is the optimal control minimizing the map $\alpha \mapsto J^{MFG}(\alpha, \hat{\xi})$ and $S^{\alpha^\star}$ is the optimal trajectory.

If additionally $\widehat \alpha_t ^\star =\hat{\xi}_t$ a.s. for all $t\in [0,T]$, then $\alpha^\star$ is a Mean-field Nash equilibrium. 

\end{Theorem}
\begin{proof}
We first show that the first implication of the above theorem holds.
\begin{align}\label{eq1}
    & \lim_{\varepsilon \to 0} \frac{J({\alpha}^{\star}+\varepsilon \beta)-J({\alpha}^{\star})}{\varepsilon}\nonumber\\&=\mathbb{E} \left[\int_0^T \left(\left(\partial_\alpha g(\alpha^{\star}_t, S^{\alpha^{\star}}_t, Q_t)+\partial_\alpha l(Q_t+ \alpha^{\star}_t) +{p_t}(\widehat{Q}_t+ {\hat \xi}_t)+J_t^{\theta}f(\widehat{\tilde Q_t} +  {\hat \xi}_t   - \bar \alpha)\right) \beta_t \right.\right.\nonumber\\
    & \left. \left.+\partial_x g(\alpha^{\star}_t, S^{\star}_t, Q_t)\bar{S}_t^\beta \right)dt  +\bar{S}_T^\beta \partial_x h({S}_T^{\alpha^{\star}}) \right],
\end{align}
with $\bar{S}_0^\beta=0$ and $\bar{S}^\beta_{t}=\int_0^t \beta_u du$.\\
By applying It\^o's formula, we get
\begin{align}\label{eq 2}
    \mathbb{E}\left[\bar{S}_T^\beta \partial_s h({S}_T^{\alpha^{\star}})\right]& =\mathbb{E}\left[\bar{S}_T^\beta {Y}_T^{\star} \right]  =\mathbb{E}\left[\int_0^T {Y}_t^{\star} d \bar{S}_t^\beta+\int_0^T \bar{S}_t^\beta d {Y}_t^{\star}\right]\nonumber \\
    & =\mathbb{E}\left[\int_0^T {Y}_t^{\star} \beta_t dt- \int_0^T \partial_x g(\alpha^{\star}_t, S^{\alpha^{\star}}_t, Q_t) \bar{S}_t^\beta dt \right].
\end{align}
From \eqref{eq1} and $\eqref{eq 2}$, we deduce
\begin{align}
\mathbb{E}\left[\int_0^T \left(Y_t^\star+\partial_\alpha g(\alpha^{\star}_t, S^{\alpha^{\star}}_t, Q_t) +\partial_\alpha l(Q_t+ \alpha^{\star}_t) +{p_t}(\widehat{Q}_t+ {\hat \xi}_t)+J_t^{\theta}f(\widehat{\tilde Q_t} +  {\hat \xi}_t-\bar \alpha)\right)\beta_t dt \right]=0.
\end{align}
By arbitrariness of $\beta \in \mathcal{A}$, we get the following coupling condition \eqref{coupling1}, i.e.
\begin{align*}
Y_t^\star+\partial_\alpha g(\alpha^{\star}_t, S^{\alpha^{\star}}_t, Q_t)+\partial_\alpha l(Q_t+ \alpha^{\star}_t) +{p_t}(\widehat{Q}_t+ {\hat \xi}_t)+J_t^{\theta}f(\widehat{\tilde Q_t} +  {\hat \xi}_t - \bar \alpha) =0.
\end{align*}
Conversely, one can easily remark  that, if the coupling condition \eqref{coupling1} is satisfied, as well as the FBSDE $\eqref{dynS}-\eqref{adj 1}$, then, by using similar arguments as above,
\begin{align*}
&\mathbb{E}\left[\int_0^T \left((\partial_\alpha g(\alpha^{\star}_t, S^{\alpha^{\star}}_t, Q_t) +\partial_\alpha l(Q_t+ \alpha^{\star}_t)+{p_t}(\widehat{Q}_t+ {\hat \xi}_t)+J_t^{\theta}f(\widehat{\tilde Q_t} +  {\hat \xi}_t-\bar \alpha))\beta_t+\partial_x g(\alpha^{\star}_t, S^{\alpha^{\star}}_t, Q_t)\bar{S}_t^\beta \right)dt \right. \nonumber \\ & \quad \left. +\bar{S}_T^\beta \partial_x h({S}_T^{\alpha^{\star}}) \right]=0.
\end{align*}
The above relation, together with $\eqref{eq1}$, implies that the Gateaux derivative of $J$  with respect to $\alpha$ is $0$ in $\alpha^\star$ and for all directions $\beta$. This result, together with the  strict convexity of $J$, allows to conclude.
\end{proof}

\paragraph{Characterization of the MFC optimal strategy.}
We now provide a characterization result for the optimal strategy of the MFC problem.

\begin{Theorem}
Let $x_0=(s_0, q_0)$ be a random vector independent of $\mathbb{F}^0$. If there exists a control $\alpha^\star \in \mathcal{A}$ which minimizes the map $\alpha \mapsto J^{MFC}(\alpha)$ and if $(S^{\alpha^\star}, Q)$ is the state process associated to the initial condition $x_0$, control $\alpha^\star$ and the dynamics \eqref{dynQ}-\eqref{dynS}, then there exists a unique solution $(Y^\star,q^{0, \star},q^{\star}, \nu^\star, \nu^{0,\star}) \in \mathcal{S}^2 \times (\mathcal{H}_2)^4$ of the BSDE with jumps
\begin{align}
    -dY^\star_t=\partial_x g(\alpha, S_t ^{\alpha^\star} , Q_t)dt-q_t^{0, \star}dW_t^0-q^\star_t dW_t - \nu^\star_t d\widetilde N_t- \nu^{0,\star}_t d\widetilde N^0_t, \nonumber \\
    Y^\star_T=\partial_x h(S_T^{\alpha^\star}),\label{adj 1-MFC}
\end{align}
satisfying the coupling condition
\begin{align}\label{coupling2}
\partial_\alpha g(\alpha^{\star}_t, S^{\alpha^{\star}}_t, Q_t)+\partial_\alpha l(Q_t+ \alpha^{\star}_t)+{p_t}(\widehat{Q}_t+ \widehat \alpha^{\star}_t)+(\pi Q_t^{st} + (1-\pi)(\widehat Q_t + \widehat \alpha^{\star}_t))\partial_\alpha p(\pi Q_t^{st} + (1-\pi)(\widehat Q_t + \widehat \alpha^{\star}_t))  \nonumber\\
 \quad \quad +Y_t^{\star} + J_t^{\theta}f(\widehat{\tilde Q_t} +  {\widehat \alpha^\star_t} - \bar \alpha)+J_t^{\theta}(\widehat{\tilde Q_t} +{\widehat \alpha^\star_t} - \bar \alpha) \partial_\alpha f(\widehat{\tilde Q_t} +  \textcolor{black}{\widehat \alpha^\star_t} - \bar \alpha) =0,
\end{align}
{\color{black}with $\widehat{\alpha}^\star$ the optional projection of $\alpha^\star$ with respect to $\mathbb F^{0}$.}

Conversely, assume that there exists $\left(\alpha^\star,S^{\alpha^\star}, Y^\star, q^{0,\star},q^{\star}, \nu^{\star}, \nu^{0,\star} \right) \in \mathcal A \times (\mathcal{S}^2)^2 \times (\mathcal{H}^2)^4$ satisfying the coupling condition $\eqref{coupling2}$, as well as the FBSDE $\eqref{dynS}-\eqref{adj 1-MFC}$, then $\alpha^\star$ is the optimal control minimizing the map $\alpha \mapsto J^{MFC}(\alpha)$ and $S^{\alpha^\star}$ is the optimal trajectory.
\end{Theorem}
\begin{Remark}\normalfont
The uniqueness is induced by the strict convexity of the criterion. The proof follows closely the proof of the theorem given for the MFG case, we give it for sake of clarity.
\end{Remark}

\begin{proof}
We start with the first implication. 

We have:
\begin{align}\label{eq3}
    &\lim_{\varepsilon \to 0} \frac{J^{MFC}({\alpha}^{\star}+\varepsilon \beta)-J^{MFC}({\alpha}^{\star})}{\varepsilon} \nonumber \\&=\mathbb{E} \left[\int_0^T \left((1-\pi)\left(\partial_\alpha g(\alpha^{\star}_t, S^{\alpha^{\star}}_t, Q_t)+\partial_\alpha l(Q_t+ \alpha^{\star}_t) +{p_t}(\widehat{Q}_t+ \widehat \alpha^{\star}_t)+J_t^\theta {f}(\widehat{Q}_t+ \widehat \alpha^{\star}_t-\bar{\alpha}) \right) \beta_t \right.\right.\nonumber\\
    &   \quad \quad + \left.\left. \left(\pi Q_t^{st} + (1-\pi)(Q_t + \alpha_t^{\star})\right)\partial_\alpha p(\pi Q_t^{st} + (1-\pi)(\widehat Q_t + \widehat \alpha^{\star}_t)) \widehat \beta_t   \right.\right. \nonumber \\ 
    & \quad \quad +\left.\left.  J_t^{\theta}({\tilde Q_t }+ \alpha_t^{\star}-\bar{\alpha})\partial_\alpha f({ \widehat {\tilde Q_t}}+ \widehat \alpha^{\star}_t-\bar{\alpha}) \widehat \beta_t + \partial_x g(\alpha^{\star}_t, S^{\alpha^{\star}}_t, Q_t)\bar{S}_t^\beta \right) dt +(1-\pi)\bar{S}_T^\beta \partial_x h({S}_T^{\alpha^{\star}})\right],
\end{align}
with $\bar{S}_0^\beta=0$ and $\bar{S}^\beta_{t}=\int_0^t \beta_u du$.\\
By applying It\^o's formula, we get
\begin{align}\label{eq 4}
    \mathbb{E}\left[\bar{S}_T^\beta \partial_s h({S}_T^{\alpha^{\star}})\right]& =\mathbb{E}\left[\bar{S}_T^\beta {Y}_T^{\star} \right]  =\mathbb{E}\left[\int_0^T {Y}_t^{\star} d \bar{S}_t^\beta+\int_0^T \bar{S}_t^\beta d {Y}_t^{\star}\right] \nonumber\\
    & =\mathbb{E}\left[\int_0^T {Y}_t^{\star} \beta_t dt- \int_0^T \partial_x g(\alpha^{\star}_t, S^{\alpha^{\star}}_t, Q_t) \bar{S}_t^\beta dt \right].
\end{align}
From \eqref{eq3} and $\eqref{eq 4}$, we deduce
\begin{align}
\mathbb{E}\left[\int_0^T \left(\partial_\alpha g(\alpha^{\star}_t, S^{\alpha^{\star}}_t, Q_t)+{p}_t(\widehat{Q}_t+ \widehat \alpha^{\star}_t)+Y_t^{\star} 
+\left(\pi Q_t^{st} + (1-\pi)\left(\widehat Q_t + \widehat \alpha^{\star}_t\right)\right)\partial_\alpha p(\pi Q_t^{st} + (1-\pi)(\widehat Q_t + \widehat \alpha^{\star}_t)) \right.\right. \nonumber \\
\left.\left. +{J_t^\theta f(\widehat{\tilde Q_t} + \widehat \alpha_t^* - \bar \alpha)} +\partial_\alpha l(Q_t+ \alpha^{\star}_t)+J_t^{\theta}(\widehat{\tilde Q_t} +  {\widehat{\alpha^\star_t}} - \bar \alpha) \partial_\alpha f(\widehat{\tilde Q_t} +  {\widehat{\alpha^\star_t}} - \bar \alpha) \right)\beta_t dt \right]=0.
\end{align}
By arbitrariness of $\beta$, we get the  coupling condition $\eqref{coupling2}$, i.e.
\begin{align}
\partial_\alpha g(\alpha^{\star}_t, S^{\alpha^{\star}}_t, Q_t)+\partial_\alpha l(Q_t+ \alpha^{\star}_t)+{p}_t(\widehat{Q}_t+ \widehat \alpha^{\star}_t)+\left(\pi Q_t^{st} + (1-\pi)\left(\widehat Q_t + \widehat \alpha^{\star}_t\right)\right)\partial_\alpha p(\pi Q_t^{st} + (1-\pi)(\widehat Q_t + \widehat \alpha^{\star}_t))  \nonumber\\
 +Y_t^{\star}+{J_t^\theta f(\widehat{\tilde Q_t} + \widehat \alpha_t^* - \bar \alpha)} +J_t^{\theta}(\widehat{\tilde Q_t} +  \alpha^\star_t - \bar \alpha) \partial_\alpha f(\widehat{\tilde Q_t} +  \alpha^\star_t - \bar \alpha) =0.
\end{align}
Conversely, similar to the MFG case, one can show that if the coupling condition is satisfied, we get that the Gateaux derivative of $J$ with respect to $\alpha$ is $0$ and we conclude by strict convexity of $J$.
\end{proof}
\paragraph{Equivalence between MFC and MFG problems.}
The two characterization systems are equivalent except the two coupling conditions. 
More precisely, let $g$, $l$, $h$ as in Assumption \ref{assump_2.1} and let ($p_{MFG}$, $f_{MFG}$) be two continuous functions with linear growth and ($p_{MFC}$, $f_{MFC}$) two $C_b^{1}$ functions satisfying the relations 
\begin{align}
\label{relation_MFG_MFC1}
  p_{MFC}(x) = \frac{\int_0^x p_{MFG}(y)dy}{x};\,\,\, f_{MFC}(x) = \frac{\int_0^x f_{MFG}(y)dy}{x}, \,\,\, x \neq 0.
\end{align}
From the two characterization results (for the MFG and MFC problems), we deduce that $\alpha^{\star}$ is a mean-field optimal control for the problem with pricing rules $p_{MFC}$ and $f_{MFC}$ if and only if it is a mean-field Nash equilibrium for the MFG problem with pricing rules $p_{MFG}$ and $f_{MFG}$. Under these assumptions, the uniqueness of the mean-field optimal control implies the uniqueness of the mean-field Nash equilibrium. In the particular case of the linear quadratic setting studied in the next section, the pricing rules corresponding to the MFG problem are given by $p_{MFG}(x)=p_0+p_1x$ and $f_{MFG}(x)=f_0+f_1x$ and there exist $C_b^1$ pricing rules for the associated mean-field control problem given by $p_{MFC}(x)=p_0+\frac{p_1}{2}x$ and $f_{MFC}(x)=f_0+\frac{f_1}{2}x$, which satisfy \eqref{relation_MFG_MFC1}. Therefore, we have uniqueness of the mean-field Nash equilibrium in this case.

Assume that $\alpha^{\star}_{MFC}$ is a mean field optimal control for the problem with pricing rules $p_{MFC}$ and $f_{MFC}$. Then, from \eqref{relation_MFG_MFC1}, we deduce that $\alpha^{\star}_{MFC}$ is a mean-field Nash equilibrium for the MFG problem with the pricing rules:
\begin{align}
\label{relation_MFG_MFC}
 & p_{MFG}(x) = p_{MFC}(x) + xp_{MFC}'(x),\\
 & f_{MFG}(x) = f_{MFC}(x) + xf_{MFC}'(x).
\end{align}
If the price function $p_{MFC}$ is increasing  and  $x>0$,  then the price used by the MFG system needs to be higher than the price used by the MFC system to produce the same result in terms of efforts $\alpha$. These two hypotheses seem quite natural as we defined $p_{MFG}$ as an inverse demand function and $x$ represents the global consumption of the Agents of the system. Indeed even if individually the power demand of an agent happens to be negative, we expect their total power demand to remain positive. We can interpret the relationship between the prices $p_{MFG}$  and $p_{MFC}$  as the additional penalisation which needs to be provided to the agents when they selfishly optimize (i.e. MFG system) to behave as if they were considering the whole community of all agents (i.e. MFC case). This is similar for the function $f_{MFG}$ which a convex growing function. It then penalizes more agents whose deviation is too low (or large) with respect to the target $\bar \alpha$ when the global deviation is also too low (or large) with respect to the target. On the contrary, it encourages agents whose deviation is on the opposite direction of the global deviation with respect to the target $\bar \alpha$ to further increase their effort in this direction as they are helping the system.

\begin{Remark}\normalfont
The relationships above between MFG and MFC are used in the numerical part in order to compute the solutions of the MFC as they allow us to use the same code for computing both equilibria at the same time.
\end{Remark}

\section{The linear-quadratic case and explicit solution of the MFG problem}
In this section, we study the linear-quadratic case (LQ case). In more detail, we provide an semi-explicit characterization of the MFG Nash equilibrium expressed through a decoupled system of forward-backward stochastic differential with jumps, involving a suitable Riccati BSDE and show that such an equilibrium provides approximate Nash equilibria in the $n$-player game for $n$ sufficiently large.   
\subsection{Main assumptions and a preliminary existence result}
We start by making the following assumption on the game coefficients.

\begin{Assumption}\label{ass-LQ}
\begin{enumerate}
\item $\mu(t,q)=\mu^{st}(t,q) = \mu q$ with $\mu \in \mathbb R$, and $\sigma(t,q)=\sigma q , \sigma^{st}(t,q)=\sigma^{st} q$, $\beta=0$ and $\beta^{st}=0$, with $\mu,\sigma,\sigma^0$ given constants.
\item $g(a,s,q)= \frac{A}{2} a^2 + \frac{C}{2} s^2 $ with $A, C \in \mathbb R_+ ^{\star}$.
\item \textcolor{black}{$l(x)= \frac{K}{2} x^2$ with $K \in \mathbb R_+$}.
\item $f(a) = f_0 + f_1 a$ with $f_i \in \mathbb R$, $i=0,1$ and $f_1 \geq 0$.
\item $p(q) = p_0 + p_1 q$ with $p_0 \in \mathbb R$, and $p_1 \in \mathbb R_+ ^{\star}$.
\item $h(s) = h_0 + h_1 s + \frac{h_2}{2} s^2$ with $h_i \in \mathbb R$, $i=0,1,2$ and $h_2 \geq 0$.
%\item \textcolor{black}{$A+(1-\pi)p_1 > 0$ and $f_1 \ge 0$.}
\end{enumerate}
\end{Assumption}

Before solving the MFG, we prove the following auxiliary existence and uniqueness result, that we will use in the next sub-sections.
\begin{Theorem}[Existence and uniqueness result for Riccati BSDEs with jumps]\label{Riccati}
Let $p \in \mathbb{R}_+$ and $\xi$ a bounded $\mathcal{F}_T^0$-measurable random variable  such that $\xi \geq 0$ a.s. Let $(\Theta_t)$ be an $\mathbb{F}^0$-adapted bounded optional stochastic process such that $\Theta_t \leq 0$ a.s. for all $t \in [0,T]$. Then there exists an unique solution $(Y,Z,U)$, with $Y$ bounded and $\mathbb{F}^0$-predictable processes $(Z,U) \in (\mathcal{H}^2)^2$  of the following Riccati BSDE with jumps:
\begin{align}\label{Ric}
-dY_t=(p+\Theta_tY_t^2)dt-Z_tdW^0_t-U_td\Tilde{N}^0_t; \,\, Y_T=\xi.
\end{align}
\end{Theorem}
\begin{proof}
\textit{Existence.} Consider the process $X^{(0)} \equiv 0$ and define for each  $n\geq 0$ the following BSDE:
\begin{align*}
    \begin{cases}
    -dX_t^{(n+1)}=\left(p+2\Theta_tX_t^{(n)}X_t^{(n+1)}-\Theta_t(X_t^{(n)})^2\right)dt-Z_t^{(n+1)}dW^0_t-U_t^{(n+1)}d\Tilde{N}^0_t; \nonumber\\
    X_T^{(n+1)}=\xi.
    \end{cases}
\end{align*}
The above BSDE is a linear BSDE with jumps, and the first component of the unique solution $(X_\cdot^{(n+1)},Z_\cdot^{(n+1)},U_\cdot^{(n+1)})$, admits the following representation:
\begin{align*}
    X_t^{(n+1)}=\mathbb{E}\left[e^{\int_t^T 2\Theta_sX_s^{(n)}ds}\xi+\int_t^Te^{\int_t^u 2\Theta_sX_s^{(n)}ds}(p-\Theta_u(X_u^{(n)})^2)du \bigg| \mathcal{F}^0_t\right].
\end{align*}
Due to the assumptions, one can easily remark that, for all $n$, $0 \leq X_t^{(n)} \leq C_n$ a.s. for all $t \in [0,T]$, with $C_n$ positive constant.

Define now the sequence: $\bar{X}^{(n)} \equiv X^{(n)}-X^{(n+1)}$, $\bar{Z}^{(n)} \equiv Z^{(n)}-Z^{(n+1)}$ and $\bar{U}^{(n)} \equiv U^{(n)}-U^{(n+1)}$, $n \geq 1$. We obtain
\begin{align*}
\begin{cases}
    -d\bar{X}_t^{(n)}=\left(2\Theta_tX_t^{(n)}\bar{X}_t^{(n)}-\Theta_t(\bar{X}_t^{(n-1)})^2\right)dt-\bar{Z}_t^{(n)}dW^0_t-\bar{U}_t^{(n)}d\Tilde{N}^0_t;\\\nonumber
    \bar{X}_T^{(n)}=0.
\end{cases}
\end{align*}

Observe that the above BSDE is linear, therefore, for $n \geq 1$, the following representation holds:
\begin{align*}
    \bar{X}_t^{(n)}=\mathbb{E}\left[\int_t^Te^{\int_t^u 2\Theta_sX_s^{(n)}ds}\left(-\Theta_u\bar{X}_u^{(n-1)})^2\right)du \bigg| \mathcal{F}^0_t\right],
\end{align*}
which leads to $\bar{X}^{(n)} \geq 0$. Consequently, the sequence $X^{(n)}$ is decreasing; since, moreover, it is bounded from below, we derive that $X^{(n)}_\cdot$ converges a.s. for all $t\in [0,T]$ to a limiting bounded process $X$, such that there exists $\mathbb{F}^0$- predictable processes $(Z,U) \in (\mathcal{H}^2)^2$  such that $(X,Z,U)$ is a solution of the Ricatti BSDE \eqref{Ric}.

\textit{Uniqueness.} Let $(Y^1,Z^1,U^1)$ and $(Y^2,Z^2,U^2)$ be two solutions of the Riccati BSDE \eqref{Ric} satisfying the integrability conditions given in the statement of the theorem. The process $(\bar{Y}, \bar{Z},\bar{U})$, with  $\bar{Y}=Y^1-Y^2$, $\bar{Z}=Z^1-Z^2$,$\bar{U}=U^1-U^2$, satisfies the linear BSDE:
\begin{align*}
    -d\bar{Y}_t=\Theta_t\bar{Y_t}(Y_t^1+Y_t^2)dt-\bar{Z}_tdW^0_t-\bar{U}_td\Tilde{N}^0_t;\,\, \bar{Y}_T=0,
\end{align*}
Observe that $(0,0,0)$ is a solution of the above linear BSDE with jumps. By uniqueness of the solution of a linear BSDE with jumps, we derive that  $\bar{Y}\equiv \bar{Z} \equiv \bar{U} \equiv 0$.
\end{proof}

\subsection{Characterization of the solution}  Now, we proceed with solving the MFG in the linear quadratic case building on the stochastic maximum principle approach developed in the previous section. Using \eqref{adj 1}, the adjoint variable is given in this case by
\[ dY_t = -CS_t ^\alpha dt + q_t ^0 dW_t ^0 + q_t dW_t + q^N _t d\widetilde N_t+q^{0,N} _t d\widetilde N^0_t, \quad Y_T = \partial_x h(S^\alpha _T)=h_1+h_2 S^\alpha _T,\]
where $dS^\alpha _t = \alpha_t dt$, $S^\alpha _0 =s_0$, and the coupling condition
\be A\alpha_t +\textcolor{black}{K(Q_t+\alpha_t)}+ p_t(\widehat Q_t + \widehat \alpha_t) +Y_t +f(\widehat{\tilde{Q}}_t+\widehat \alpha_t - \bar \alpha)J_t ^\theta =0, \label{coupling}\ee
where
\[ J_t ^\theta = \textbf{1}_{R_t \leq \theta}, \quad t\in [0,T].
\]
 
Notice that $J^\theta$ above is independent of the control $\alpha$.
Since $W,W^0,N^0$ are independent, hypothesis ($\mathcal H$) in \cite[Section 2.4]{bremaud1978changes} is fulfilled, hence using Proposition 7(ii) in \cite{bremaud1978changes} and exploiting the fact that $J^\theta$ is already $\mathbb F ^{0}$-adapted, we can project the adjoint variables $(\widehat Y, \widehat q^0, \widehat q^{0,N})$ into the filtration $\mathbb F^0$ leading to the following BSDE with jumps:
\be d\widehat Y_t = -CS_t ^{\widehat \alpha} dt + \widehat q_t ^0 dW_t ^0 + \widehat q^{0,N} _t d\widetilde N^0_t, \quad \widehat Y_T = \partial_x h(S^{\widehat \alpha} _T), \label{adjoint}
\ee
while the coupling condition becomes
\be A\widehat \alpha_t +\textcolor{black}{K(\widehat{Q}_t+\widehat{\alpha}_t)}+ p_t(\widehat Q_t + \widehat \alpha_t) +\widehat Y_t +f(\widehat{\tilde{Q}}_t+\widehat \alpha_t - \bar \alpha)J_t ^\theta =0, \label{proj-coupling}\ee
which in this LQ case reads (also recall \eqref{price_fct}) as
\[ A\widehat \alpha_t+\textcolor{black}{K(\widehat Q_t + \widehat \alpha_t)} + p_0 + p_1(\pi Q_t^{st} + (1-\pi)(\widehat Q_t +\widehat \alpha_t)) +\widehat Y_t +(f_0 +f_1(\widehat{\tilde{Q}}_t+\widehat \alpha_t - \bar \alpha))J_t ^\theta =0.\]
We look for a solution taking the form:
\[\widehat{Y}_t=\bar{\phi_t}S_{t}^{\widehat{\alpha}}+\bar{\psi}_t,\]
with $\bar \phi=(\bar{\phi}_t)_{t\in [0,T]}$ and $\bar \psi = (\bar{\psi}_t)_{t\in [0,T]}$  obtained by solving the following two BSDEs with jumps:
\be d\bar{\phi}_t=\left(-C+\frac{1}{A+\textcolor{black}{K}+(1-\pi)p_1+f_1J_t^{\theta}}\bar{\phi}_t^2\right) dt+ \widehat \xi^0_t dW_t ^0 + \widehat \xi^{0,N}_t  d\widetilde N^0_t, \quad \bar{\phi}_T=h_2.\label{phihat} \ee
and
\begin{align}
\label{psihat}
\begin{cases}
d\bar \psi_t = \frac{\bar{\phi}_t}{A+(1-\pi)p_1+\textcolor{black}{K}+f_1J_t^{\theta}}\left[p_0+\pi p_1 Q_t^{st} + \textcolor{black}{((1-\pi)p_1+K)}\widehat Q_t+J_t^{\theta}(f_0+f_1(\widehat{Q_t}-\mathbb{E}[\widehat Q_t]-\bar{\alpha}))+\bar \psi_t \right]dt \\
\quad + \widehat{\eta}_t^0 dW_t ^0 +  \widehat \eta^{0,N}_t  d\widetilde N^0_t \\
\bar \psi_T = h_1,
\end{cases}
\end{align}
which are derived by using the ansatz and It\^o's formula. Note that the two above BSDEs admit an unique solution (see Theorem \ref{Riccati} for the Riccati BSDE with jumps and Theorem 2.4 in \cite{royer2016} for the linear BSDE with jumps).\\
One can check with It\^{o}'s formula, using the ansatz and the (projected) coupling condition that given  $(\bar \phi, \widehat \xi^0,\widehat \xi^{0,N})$   the unique solution of the Riccati BSDE (\ref{phihat}) and $(\bar \psi, \widehat{\eta}^0, \widehat \eta^{0,N})$ the unique solution of the BSDE (\ref{psihat}), then the triple ($\Tilde Y$, $\Tilde q ^0$, $\Tilde q^{0,N}$) defined by
\[\Tilde{Y}_t:=\bar{\phi_t}S_{t}^{\widehat{\alpha}}+\bar{\psi}_t, \quad \Tilde q_t ^0 := \widehat \xi^0_tS_{t}^{\widehat{\alpha}}+\widehat{\eta}_t^0, \quad \Tilde q_t ^{0,N} := \widehat \xi^{0,N}_t S_{t}^{\widehat{\alpha}}+\widehat{\eta}_t^{0,N}, \quad t \in [0,T],\]
is a solution of the adjoint equation (\ref{adjoint}). Since equation (\ref{adjoint}) admits an unique solution, we get $\Tilde{Y}=\widehat{Y}$, $\Tilde q^0=\widehat q^0$ and $\Tilde q^{0,N}=\widehat q^{0,N}$.
 Therefore, by using again the coupling condition and by substituting the ansatz in the projected coupling condition,  $ \widehat{\alpha}_t$ has the following representation in feedback form:
\be \widehat{\alpha}_t = - \frac{1}{K_t^{\theta}}\left(p_0+\pi p_1 Q_t^{st} +\textcolor{black}{((1-\pi)p_1+K)}\widehat Q_t + \bar{\phi_t}S_{t}^{\widehat{\alpha}}+\bar{\psi}_t + J_t^{\theta}\left\{f_0 +f_1(\widehat{Q}_t-\mathbb{E}(\widehat{Q}_t)- \bar \alpha)\right\} \right), \label{alphahat}\ee
where $K_t ^\theta := A+(1-\pi)p_1 +\textcolor{black}{K}+ f_1 J_t ^\theta$, which is strictly positive due Assumption \ref{ass-LQ}.6. Here the expression for $ \widehat{\alpha}_t$ is linear in $\widehat S_t := S_t^{\widehat \alpha}$, so $\widehat S$ satisfies a linear SDE that can be solved explicitly, yielding
\be 
\widehat S_t = s_0 e^{-\int_0^t\bar \phi(u)/K_u ^\theta du}+\int_0^t \widehat A_r  e^{-\int_r^t\bar \phi(s)/K_s^{\theta} ds} dr
\label{Shat} \ee where $\widehat A_r=-\frac{1}{K_r^{\theta}}\left(p_0+\pi p_1 Q_r^{st} +((1-\pi)p_1+\textcolor{black}{K)}\widehat Q_r +\bar{\psi}_r + J_r^{\theta}(f_0+f_1(\widehat{\tilde Q_r} -\bar{\alpha})) \right)$. With $\widehat S_t$, we can obtain $ \widehat{\alpha}_t$ by (\ref{alphahat}).

With $\widehat \alpha$, we go back to the (unprojected) coupling condition to find $\alpha^{\star}$. By assuming $Y_t=\phi_t S_{t}^{\alpha}+\psi_t$ and proceeding in the same way as for $\widehat Y_t$, we have
\be d\phi_t=\left(-C+\textcolor{black}{\frac{1}{A+K}}\phi_t^2\right) dt, \quad \phi_T=h_2, \label{phi}\ee
and
\be d \psi_t&=&
\textcolor{black}{\frac{\phi_t}{A+K}}\left[\textcolor{black}{KQ_t}+p_0+\pi p_1 Q_t^{st} +(1-\pi)p_1(\widehat Q_t+\widehat \alpha_t )+J_t^{\theta}(f_0+f_1(\widehat{\tilde Q}_t +\widehat \alpha_t -\bar{\alpha}))+ \psi_t\right]dt \nonumber\\
&&+ \eta_t ^0 dW_t ^0 + \eta_t dW_t+ \eta_t ^{0,N} d\widetilde N^0_t, \label{psi}\ee
with  $\psi_T=h_1$. Finally, we can provide an expression for $\alpha^{\star}$ as follows:
\begin{align}
\alpha^{\star} _t & = \textcolor{black}{\frac{1}{A+K}}\left(\textcolor{black}{-KQ_t-}p_0  - \pi p_1 Q_t^{st} - p_1(1-\pi) (\widehat Q_t + \widehat \alpha_t) \right. \nonumber \\ 
& \quad \left. - \phi_t S^{\alpha^{\star}} _t -  \psi_t -\left\{f_0 +f_1\left(\widehat{\tilde Q}_t + \widehat \alpha_t- \bar \alpha\right)\right\}J_t ^\theta   \right),
\label{alpha} \end{align}
which is linear in $S^{\star} _t = S_t ^{\alpha^{\star}}$ as in the projected case, so the SDE for $S$ controlled by $\alpha^{\star}$ can be solved explicitly as
\be
S^{\star} _t = e^{-\int_0 ^t (\phi_u /\textcolor{black}{(A+K)}) du } \left( s_0 + \int_0 ^t A^{\star} _r e^{\int_0 ^r (\phi_u /\textcolor{black}{(A+K)}) du } dr \right),
\label{S} \ee
where
\[ A_r ^{\star} = \textcolor{black}{\frac{1}{A+K}} \left(\textcolor{black}{-KQ_r-}p_0 -\pi p_1 Q_r^{st} - (1-\pi)p_1 (\widehat Q_r + \widehat \alpha_r) -  \psi_r - \left(f_0 + f_1\left(\widehat{\tilde Q}_r + \widehat \alpha_r- \bar \alpha\right) J_r ^\theta \right) \right).\]
Finally, exploiting the integrability properties of $\phi$ and $\psi$ and the boundedness of $J^\theta$, one can easily prove that $\alpha^\star$ belongs to $\mathcal A$.

\subsection{Approximate Nash equilibria for the n-player model}

We describe here how the equilibrium strategy $\alpha^{\star}$ found above can be implemented in the $n$-player game in order to obtain an approximate Nash equilibrium $(\alpha^{\star,1}, \ldots, \alpha^{\star,n})$ with vanishing error as $n \to \infty$. Let $i=1,\ldots,n$ and let
\begin{align} \label{eps-nash}
\alpha^{\star,i} _t & = \frac{1}{A+K}\left(-K Q^{i}_t -p_0 -\pi p_1 Q_t^{st}-p_1(1-\pi)(\widehat Q_t + \widehat \alpha_t) \right. \nonumber \\
& \quad \left. - \phi_t S^{\star,i} _t -  \psi^{i}_t -(f_0 +f_1 (\widehat Q_t - \mathbb E[\widehat Q_t] - \bar \alpha)J_t ^\theta   \right),
\end{align}
with $\widehat \alpha_t$ and $J_t ^\theta$ as above, while $\phi$ and $(\psi^i ,\eta^{0,i}, \eta^{N,i})$, $i=1,\ldots,n$, are solutions to the following BSDE
\be d\phi_t=\left(-C+\frac{1}{A+K} \phi_t ^2\right) dt
, \quad \phi_T=h_2, \ee
and
\begin{align} d \psi^i _t &=  \frac{\phi_t}{A+K}\left[K Q^i _t + p_0+\pi p_1 Q_t^{st} +(1-\pi)p_1(\widehat Q_t+\widehat \alpha_t )+J_t^{\theta}(f_0+f_1(\widehat{Q}_t +\widehat \alpha_t -\mathbb{E}[\widehat Q_t]-\bar{\alpha}))+ \psi^i _t\right]dt \nonumber\\
& \quad + \eta_t ^{0,i} dW_t ^0 + \eta_t ^{i} dW_t ^{i} + \eta_t ^{N,i} d\widetilde N^0_t  ,  \end{align}
with  $\psi^i _T=h_1$, where $\widehat \alpha$ is as in \eqref{alphahat}.
Theorem \ref{Riccati} and Theorem 2.4 in \cite{royer2016} grant existence and uniqueness of the solution for system above. A useful consequence of the definition of $\alpha^{\star,i}$ is that the two vector-valued processes
\be
(\alpha^\star, S^\star, Q, \widehat Q, Q^{st}) \quad \textrm{and} \quad (\alpha^{\star,i}, S^{\star,i}, Q^i , \widehat Q, Q^{st}) \label{equa-distr} 
\ee
have the same distribution, for all $i=1,\ldots,n$.
The description of $i$-th player strategy $\alpha^{\star,i}$ is completed by the expression for $S^{\star,i}$ given by 
\be
S^{\star,i} _t = s_0 + \int_0 ^t \alpha^{\star,i}_u du = e^{-\int_0 ^t (\phi _u /(A+K)) du } \left( s_0 + \int_0 ^t A^{i} _r e^{\int_0 ^r (\phi _u /(A+K)) du } dr \right), \label{expl-S}
\ee
where
\[ A_r ^{i} = \frac{1}{A+K} \left(-K Q^{i}_t -p_0 - \pi p_1 Q_r^{st} - (1-\pi)p_1 (\widehat Q_r + \widehat \alpha_r) -  \psi^i _r - (f_0 + f_1 (\widehat \alpha_r - \bar \alpha)) J_r ^\theta \right).\]
Finally, we recall that
\[ d\widehat Q_t = \mu \widehat Q_t dt + \sigma^0 \widehat Q_t dW_t ^0 ,\]
and that $\mathbb E[Q_t]=\mathbb E[\widehat Q_t] = Q_0 e^{\mu t}$ for all $t \in [0,T]$.

\begin{Proposition}
The strategy profile $(\alpha^{\star,1}, \ldots, \alpha^{\star,n})$ defined in \eqref{eps-nash} is an $\varepsilon_n$-Nash equilibrium with $\varepsilon_n \to 0$ as $n \to \infty$. 
\end{Proposition}

\begin{proof} The proof consists in showing the following two limits for $i=1,\ldots,n$:
\begin{enumerate}
    \item $\lim_{n \to \infty} J_n ^i (\alpha^{\star,1},\ldots,\alpha^{\star,n}) = J^{MFG}(\alpha^{\star})$;
    \item $\limsup_{n \to \infty} \sup_{\alpha^i \in \mathcal A} J_n ^i (\alpha^i , \alpha^{\star,-i}) \geq J^{MFG}(\alpha^{\star})$.
\end{enumerate}
Combining the two limits above we would get that for all $\varepsilon>0$ there exists $n_\varepsilon \in \mathbb N$ such that
\[ J_n ^i (\alpha^{\star,1},\ldots, \alpha^{\star,n}) \ge \sup_{\alpha^i} J_n ^i (\alpha^i , \alpha^{\star,-i}) - \varepsilon,\]
for all $i = 1,\ldots,n$, i.e. $(\alpha^{\star,1}, \ldots, \alpha^{\star,n})$ is an $\varepsilon$-Nash equilibrium, and for all $n \ge n_\varepsilon$.
By symmetry, it clearly suffices to prove properties 1 and 2 above only for the first player. Recall the following two expressions for the MFG objective functional, where the second one comes from \eqref{equa-distr} (with $i=1$) together with projecting onto $\mathcal F_t ^0$,

\begin{align*}
J^{MFG}(\alpha^{\star}) & = \mathbb E\left[\int_0 ^T \left(\frac{A}{2}(\alpha_t ^{\star})^2 + \frac{C}{2}(S_t ^{\star})^2 + \frac{K}{2}(Q_t + \alpha_t ^\star)^2 \right. \right.  \\
& \quad \quad  \left. \left. + (Q_t + \alpha_t ^{\star})(p_0 +p_1 (\pi Q_t^{st} + (1-\pi)(\widehat Q_t + \widehat{\alpha}_t ^{\star})))\right. \right. \\
& \quad \quad \left.\left. + J_t ^\theta (\tilde Q_t + \alpha_t ^{\star} - \bar \alpha)(f_0 + f_1 (\widehat{\tilde Q}_t + \widehat{\alpha}^{\star} _t - \bar \alpha))\right) dt  + h_0 + h_1 S_T ^{\star} + \frac{h_2}{2}(S_T ^{\star})^2 \right] \\
& = \mathbb E\left[\int_0 ^T \left(\frac{A}{2}(\alpha_t ^{\star,1})^2 + \frac{C}{2}(S_t ^{\star,1})^2 + \frac{K}{2}(Q^1 _t + \alpha_t ^{\star,1})^2 \right. \right. \\
& \quad \quad \left. \left. +  (Q^1 _t + \alpha_t ^{\star,1})(p_0 +p_1 (\pi Q_t^{st} + (1-\pi)({\widehat Q}^1_t + \widehat \alpha_t ^{\star,1})))\right. \right. \\
& \quad \quad \left.\left. + J_t ^\theta (\tilde{Q}^1 _t + \alpha_t ^{\star,1} - \bar \alpha)(f_0 + f_1 (\widehat{\tilde Q}^1_t  + \widehat{\alpha}^{\star,1} _t - \bar \alpha))\right) dt + h_0 + h_1 S_T ^{\star,1} + \frac{h_2}{2}(S_T ^{\star,1})^2 \right]
\end{align*}

and
\begin{align*}
J^1 _n (\alpha^1, \alpha^{\star,-1}) & = \mathbb E\left[\int_0 ^T \left(\frac{A}{2}(\alpha_t ^1 )^2 + \frac{C}{2}(S_t ^{\alpha^1})^2 + \frac{K}{2}(Q^1 _t + \alpha_t ^{1})^2 \right. \right. \\
& \quad \quad \left. \left. + (Q^1 _t + \alpha_t ^1)\left(p_0 +p_1 \left(\pi Q_t ^{st} + (1-\pi) \frac{1}{n}\sum_{j} \left(Q_t ^j + \alpha_t ^{\star,j}\right)\right)\right)\right. \right. \\
& \quad \quad \left.\left. + J_t ^\theta (\tilde Q^1 _t + \alpha_t ^1 - \bar \alpha)\left(f_0 + f_1 \left(\frac{1}{n}\sum_{j} (\widehat{\tilde Q_t }^j + \alpha^{\star,j} _t) - \bar \alpha \right)\right)\right) dt \right. \\
& \quad \quad \left. + h_0 + h_1 S_T ^{\alpha^1} + \frac{h_2}{2}(S_T ^{\alpha^1})^2 \right],
\end{align*}
where we recall that $\tilde{Q_t} = Q_t - \mathbb E[Q_t] = Q^1 _t - \mathbb E[Q^1 _t]$.\medskip

\noindent 1. Let us show the first limit. Using the expressions above and the fact that $\mathbb E[Q_t] = \mathbb E[Q^1 _t]$ for all $t\in [0,T]$, we obtain
\begin{align}
  J^{MFG}(\alpha^{\star}) -  J^1 _n (\alpha^{\star,1}, \alpha^{\star,-1}) & = \mathbb E\left[\int_0 ^T \left( p_1(1-\pi) (Q^1 _t + \alpha_t ^{\star,1}) \left( (\widehat Q^1 _t + \widehat \alpha_t ^{\star,1}) - \frac{1}{n}\sum_{j} \left(Q_t ^j + \alpha_t ^{\star,j}\right)\right) \right.\right. \nonumber \\
  & \hspace{1cm} \left.\left. + f_1 J_t ^\theta \left( \tilde Q^1 _t + \alpha^{\star,1} _t -\bar\alpha \right)\left( (\widehat{\tilde Q}^1 _t + \widehat{\alpha}^{\star,1} _t) - \frac{1}{n}\sum_{j} (Q_t ^j + \alpha^{\star,j} _t)\right)\right) dt \right] .
\end{align}
Now, since $J_t ^\theta$ is uniformly bounded, there exists a constant $C>0$ (changing possibly from line to line) such that
\begin{align}
  | J^{MFG}(\alpha^{\star}) -  J^1 _n (\alpha^{\star,1}, \alpha^{\star,-1}) | & \le C \sup_{t\in [0,T]} \left\|Q^1 _t + \alpha^{\star,1} _t \right \|_{L^2} \left \| \int_0 ^T \left((\widehat Q^1 _t + \widehat \alpha_t ^{\star,1}) - \frac{1}{n}\sum_{j} \left(Q_t ^j + \alpha_t ^{\star,j}\right)\right) dt \right \|_{L^2} \\
 & \le  C \sup_{t\in [0,T]} \left\|Q _t + \alpha^{\star} _t \right \|_{L^2} \sup_{t \in [0,T]} \left \| (\widehat Q^1 _t + \widehat \alpha_t ^{\star,1}) - \frac{1}{n}\sum_{j} \left(Q_t ^j + \alpha_t ^{\star,j}\right) \right \|_{L^2},
\end{align}
where we observe that $\sup_{t\in [0,T]} \left\|Q _t + \alpha^{\star} _t \right \|_{L^2}$ is finite. Hence, to show to conclude this part it suffices to show
\begin{equation}
    \sup_{t \in [0,T]} \left \| \widehat Q^1 _t - \frac{1}{n}\sum_{j} Q_t ^j \right \|_{L^2} + \sup_{t \in [0,T]} \left \| \widehat \alpha_t ^{\star,1} - \frac{1}{n}\sum_{j} \alpha_t ^{\star,j} \right \|_{L^2} \to 0, \quad n \to \infty.
\end{equation}
Now, the convergence to zero of the first summand on the LHS follows from conditional propagation of chaos applied to the sequence of processes $Q^i$.\footnote{See, for instance, Theorem 2.12 in Carmona and Delarue book \cite{CarmonaDelarueBook17}, Vol. II, which can be easily extended to our case with jumps exploiting in particular the fact that the intensities are constant.} It remains to show the convergence to zero of the second summand. Using the expressions for $\alpha^{\star,i}$, $i=1\ldots, n$, and since both $\phi$ and $J^\theta$ are bounded, we have
\begin{equation*}
\left \| \widehat \alpha_t ^{\star,1} - \frac{1}{n}\sum_{j} \alpha_t ^{\star,j} \right \|_{L^2} \le C^\prime \left(\left \| \widehat Q^1 _t - \frac{1}{n}\sum_{j} Q_t ^j \right \|_{L^2} \!\! + \left \| S^{\star,1} _t - \frac{1}{n}\sum_{j} S_t ^{\star,j} \right \|_{L^2} \!\! + \sup_{t \in [0,T]} \left \| \psi^1 _t - \frac{1}{n}\sum_{j} \psi_t ^j \right \|_{L^2}\right),   
\end{equation*}
for some further constant $C^\prime >0$, so we are left with showing
\[ \sup_{t \in [0,T]} \left \| \psi^1 _t - \frac{1}{n}\sum_{j} \psi_t ^j \right \|_{L^2}+\sup_{t \in [0,T]} \left \| S^{\star,1} _t - \frac{1}{n}\sum_{j} S_t ^{\star,j} \right \|_{L^2} \to 0,\quad n \to \infty.\]
The limit
\begin{equation} \sup_{t \in [0,T]} \left \| \psi^1 _t - \frac{1}{n}\sum_{j} \psi_t ^j \right \|_{L^2} \to 0, \quad n \to \infty,\label{conv-L2-psi}\end{equation}
can be obtained as follows: since it satisfies a linear BSDE, the process $\psi^i$ can be represented as
\[ \psi_t ^i = \mathbb E \left[ \frac{\Gamma_T}{\Gamma_t} h_1 + \int_t ^T \frac{\Gamma_s}{\Gamma_t} C_s ^i ds \Big | \mathcal F_t \right],\quad t \in [0,T],\]
where we set
\begin{align*}
d\Gamma_t = &-\Gamma_t \frac{\phi_t}{A+K} dt, \quad \Gamma^0 =1,\\
C_t ^i = & -\frac{\phi_t}{A+K} \left[ KQ_t ^i + p_0 + \pi p_1 Q_t^{st}+ (1-\pi)p_1 (\widehat Q_t + \hat \alpha_t) + J_t ^\theta (f_0 + f_1 (\widehat{\tilde Q_t} + \widehat \alpha_t - \bar \alpha )) \right]. 
\end{align*}
Therefore, we get
\[ \frac{1}{n}\sum_i \psi_t ^i = \mathbb E \left[ \frac{\Gamma_T}{\Gamma_t} h_1 + \int_t ^T \frac{\Gamma_s}{\Gamma_t} \left( \frac{1}{n}\sum_i C_s ^i \right)ds \Big | \mathcal F_t \right],\]
so that, using the fact that $\Gamma$ is bounded,
\[ \left| \psi_t ^1 - \frac{1}{n}\sum_i \psi_t ^i \right| \le C \mathbb E\left[ \sup_{s\in [0,T]} \left |Q_s ^1 - \frac{1}{n}\sum_i Q_s ^i\right| \Big | \mathcal F_t\right],  \] for some constant $C>0$. Hence, from above we can deduce the convergence \eqref{conv-L2-psi} from the analogue for $Q^i$.
We now exploit \eqref{expl-S} to get the following estimate
\[ \sup_{t \in [0,T]} \left \| S^{\star,1} _t - \frac{1}{n}\sum_{j} S_t ^{\star,j} \right \|_{L^2} \le C'' \left( \sup_{t \in [0,T]} \left \| \widehat Q^1 _t - \frac{1}{n}\sum_{j} Q_t ^j \right \|_{L^2} + \sup_{t \in [0,T]} \left \| \psi^1 _t - \frac{1}{n}\sum_{j} \psi_t ^j \right \|_{L^2} \right), \]
for some constant $C'' >0$.  This concludes the proof of item 1.

Finally, using similar arguments we can also show the limit in item 2., i.e.
\[ \limsup_{n \to \infty} \sup_{\alpha^i \in \mathcal A} J_n ^i (\alpha^i , \alpha^{\star,-i}) \geq J^{MFG}(\alpha^{\star}).\]
The final statement is a straightforward consequence of items 1 and 2 as already explained at the beginning of this proof.
\end{proof}

\section{Numerical approach and results}
In this section, we describe the methodology used for the numerical approximation of the  \textit{mean-field game equilibrium} and the \textit{mean-field game optimal control}  in the linear-quadratic case and provide some numerical results and interpretations.
\subsection{Numerical implementation} We propose an implementable numerical scheme which is based on the approach introduced in \cite{lejay2014} (also used in e.g. \cite{dl2016}, \cite{dl2016-1}) for the approximation of the solution of a Lipschitz BSDE, driven by a Brownian motion and a compensated Poisson process. For sake of clarity, we describe first the method used for the approximation of the optional projection of the optimal control, i.e. $\widehat{\alpha}$. More specifically, this method is based on the approximation of the Brownian motion $W^0$ and the compensated Poisson process $\widetilde{N}^0$ by two independent random walks. For $n\in \mathbb{N}$, we introduce the first random walk $\{W_k^{0,n}: k=0,\ldots ,n\}$ which is given by
\[W_0^{0,n} = 0, \quad W_k^{0,n} = \frac{1}{\sqrt{n}}\sum_{i=1}^k \epsilon_i^n \quad k=1,\ldots,n,\]
where $\epsilon^n_1,\ldots,\epsilon^n_n$ are independent symmetric Bernoulli random variables:
\[\mathbb{P}(\epsilon^n_k =1) = \mathbb{P}(\epsilon^n_k =-1) = 1/2, \quad k=1,\ldots,n.\]
The second random walk $\{\widetilde{N}_k^{0,n}: k=0,\ldots,n\}$ is non symmetric, and given by
\[\widetilde{N}_0^{0,n}=0, \widetilde{N}_k^{0,n} = \sum_{i=1}^k \eta_i^n  \quad k=1,\ldots,n,\]
where $\eta_1^n, \ldots,  \eta_n^n$ are independent and identically distributed random variables with probabilities, for each $k$, given by
\[\mathbb{P}(\eta_k^n =\kappa_n-1) = 1-\mathbb{P}(\eta_k^n =\kappa_n) = \kappa_n, \quad k=1,\ldots,n,\]
with $\kappa_n=e^{-\frac{\lambda}{n}}$. We assume that both sequence $\epsilon_1^n,\ldots,\epsilon_n^n$ and $\eta_1^n,\ldots,\eta_n^n$ are defined on the original probability space $(\Omega,\mathcal{F},\mathbb{P})$ (that can be enlarged if necessary), and that they are mutually independent. The (discrete) filtration in the probability space is $\mathbb{F}^n=\{\mathcal{F}^n_k:k=0,\ldots,n\}$ with $\mathcal{F}^n_0 = \{\Omega,\emptyset\}$ and $\mathcal{F}^n_k = \sigma\{\epsilon^n_1,\ldots,\epsilon^n_k,\eta_1^n, \ldots,  \eta^n_k\}$ for $k=1,\ldots,n$. In the discrete stochastic basis, given an $\mathcal{F}_{k+1}^n$ random variable $y_{k+1}$, in order to represent the martingale difference $\mu_{k+1}:=y_{k+1}-\mathbb{E}(y_{k+1}|\mathcal{F}_k^{n})$ we need a set of three orthogonal martingales. For this reason, we introduce a third martingale increments sequence $\{\mu_k^n:=\epsilon_k^n \eta_k^n : k=0,\ldots,n\}$. In this setting, there exist unique $z_k, u_k, v_k$ such that
\begin{align}\label{REF}
\mu_{k+1}=y_{k+1}-\mathbb{E}(y_{k+1}|\mathcal{F}_k^{n})=\frac{1}{\sqrt{n}}z_{k}e_{k+1}^n+u_{k}\eta_{k+1}^n+v_k \mu^n_{k+1}.
\end{align}
We now illustrate our numerical algorithm.

\paragraph{Numerical computation of $\widehat{\alpha}$.} Let $\delta:=\frac{T}{n}$, where $n$ represents the number of time discretization points  and let $\{ t_j=j\delta;\,\,\,j=0,\ldots,n$\}.
\begin{itemize}
\item [\underline{Step 1}] As in \cite{lejay2014} Section 4 (see also \cite{slominski1989}), we simulate  the discrete time SDEs $\widehat{q}^n$, $\widehat{q}^{st,n}$ and $R^n$ (which converge to $\widehat{Q}$,  $\widehat{Q}^{st}$ and $R$ in probability in the $J_1$-Skorokhod topology) as follows: Set $q_0^n=q_0$ (resp. $q_0^{st,n}=q_0^{st,n}$ and $R_0^n=R_0$) and for $i=0,\ldots,n,$ define the discrete time SDEs:
\begin{align}\label{discreteSDE}
\begin{cases}
    \widehat{q}_{i+1}^n=\widehat{q}_{i}^n+\delta \mu \widehat{q}_{i}^n+\sqrt{\delta} \sigma \widehat{q}_{i}^n e_{i+1}^n;\\
    \widehat{q}_{i+1}^{st,n}=\widehat{q}_{i}^{st,n}+\delta \mu^{st} \widehat{q}_{i}^{st,n}+\sqrt{\delta} \sigma^{st} \widehat{q}_{i}^{st,n} e_{i+1}^n;\\
    R_{i+1}^n=R_{i}^n+\delta (1+\lambda^0) R_{i}^n+R_{i}^n \eta_{i+1}^n,
\end{cases}
\end{align}
where $q_i:=q_{t_i}$, for $i=0,\ldots,n$ (this notation is adopted for all discrete time processes).

\item [\underline{Step 2}] 
Since the unique solution $\bar{\phi}$ of the Riccati BSDE \eqref{phihat} is positive bounded (see Theorem \ref{Riccati}), it can be approximated by a discrete time BSDE with jumps $\widetilde{\phi}^{n}$, using the  algorithm proposed in \cite{lejay2014} for Lipschitz BSDEs with jumps:

\begin{itemize}
 \item [$\bullet$] Set the terminal condition $\widetilde{\phi}^{n}_n = h_2$.
 \item [$\bullet$] For $k$ from $n-1$ down to 0, solve the discrete time BSDE
 \begin{align*}
 \widetilde{\phi}^{n}_k &=\widetilde{\phi}^{n}_{k+1} + \delta \left[C-\frac{1}{A+(1-\pi)p_1+K+f_1J_k^{\theta,n}}\left(\widetilde{\phi}_k^{n}\right)^2\right] \nonumber \\&-\sqrt{\delta} z_k^{n} e_{k+1}^n-u_k^n \eta_{k+1}^n-v_k^n \mu_{k+1}^n,
 \end{align*}
 with $J_k^{\theta,n}=\textbf{1}_{R^n_k \leq \theta}$,\,\,$k=\overline{0,\ldots,n}$.
 
 In view of the representation \eqref{REF}, the above equation is equivalent to:
 \begin{align}\label{eq}
 \widetilde{\phi}^{n}_k &=\mathbb{E}\left(\widetilde{\phi}^{n}_{k+1}\big| \mathcal{F}_k^n\right) \nonumber \\ &  +\delta\left[C-\frac{1}{A+(1-\pi)p_1+K+f_1J_k^{\theta,n}}\left(\widetilde{\phi}_k^{n}\right)^2\right]
 \end{align}
 and $\widetilde{\phi}^{n}_k$ can be obtained by a fixed point principle. At each time step, one needs to
 compute $\mathbb{E}\left(\widetilde{\phi}^{n}_{k+1}\big| \mathcal{F}_k^n\right)$, which is done using the formula
 \be
\mathbb{E}\left(\widetilde{\phi}^{n}_{k+1}\big| \mathcal{F}_k^n\right) &=& \frac{\kappa_n}{2} \widetilde{\phi}^{n}_{k+1}(\epsilon^n_1,...,\epsilon^n_k,1,\eta_1^n, \ldots,  \eta^n_k, \kappa_n-1)\nonumber\\
&+& \frac{\kappa_n}{2} \widetilde{\phi}^{n}_{k+1}(\epsilon^n_1,\ldots,\epsilon^n_k,-1,\eta_1^n, \ldots,  \eta^n_k, \kappa_n-1)\nonumber\\
&+& \frac{1-\kappa_n}{2} \widetilde{\phi}^{n}_{k+1}(\epsilon^n_1,\ldots,\epsilon^n_k,1,\eta_1^n, \ldots,  \eta^n_k, \kappa_n)\nonumber\\
&+& \frac{1-\kappa_n}{2} \widetilde{\phi}^{n}_{k+1}(\epsilon^n_1,\ldots,\epsilon^n_k,-1,\eta_1^n, \ldots,  \eta^n_k, \kappa_n).\nonumber
 \ee

\end{itemize}

\item [\underline{Step 3}] Approximate the solution $\bar{\psi}$ of the linear BSDE (\ref{psihat}) by a discrete time BSDE with jumps $\widetilde{\psi}^{n}$, using the same algorithm as in \textit{Step 2}. 
\begin{align*}
\begin{cases}
\widetilde \psi_n^{n} &= h_1\\
\widetilde \psi^{n}_{i} &= \mathbb{E}(\widetilde \psi^{n}_{i+1}|\mathcal{F}_i^{n})-\delta \frac{\widetilde{\phi}^{n}_i}{A+(1-\pi)p_1+K+f_1J_{i}^{n,\theta}}\left[p_0+\pi p_1 \widehat{q}_{i}^{st,n} + ((1-\pi)p_1+K)\widehat {q}^n_i\right.\\ & \left. + J_i^{n,\theta}(f_0+f_1(\widehat{q}^n_i-\mathbb{E}[\widehat q^n_i]-\bar{\alpha}))+\widetilde \psi^{n}_i\right],\,\, i=\overline{n-1,\ldots,0}. \\
\end{cases}
\end{align*}
\item [\underline{Step 4}] Compute $\widehat{S}^{n}_{i}$ as follows
\begin{align*}
\begin{cases} 
\widehat S^{n}_{0} & = s_0 \\
\widehat{S}^{n}_{i+1}&=\widehat{S}^{n}_{i}+\delta \left[\widehat{S}^{n}_{i}\widetilde{\phi}^{n}_i+\widetilde \psi^{n}_{i}-\frac{1}{K_{i}^{n,\theta}}\left(p_0+\pi p_1 \widehat{q}_{i}^{st,n} + ((1-\pi)p_1+K)\widehat {q}^n_i\right. \right.\\ &\left. \left. +J_i^{n,\theta}(f_0+f_1(\widehat{q}^n_i-\mathbb{E}[\widehat q^n_i]-\bar{\alpha}))  \right)\right], \,\, i=\overline{0,\ldots,n-1},
\end{cases}
\end{align*}
with $K_{i}^{n, \theta}= A+(1-\pi)p_1+K+f_1J_{i}^{n,\theta},\,\,i=\overline{0,\ldots,n}$.

\item [\underline{Step 5}] Compute $\widehat{\alpha}^{n}_i$, for $i=0,\ldots,n$ by 
\begin{align*}
\widehat \alpha^{n}_i &= -\frac{1}{K_i^{n,\theta}} \left( p_0 + \pi p_1\widehat q^{st,n}_i + ((1-\pi)p_1+K)\widehat q^n_i + \widetilde \phi^{n}_i \widehat{S}_i ^{n} +  \widetilde \psi^n _i \right. \nonumber \\
& \left.\hspace{2cm} + \left\{f_0 +f_1(\widehat{q}^n_i-\mathbb{E}(\widehat{q}^n_i)- \bar \alpha)\right\}J_i ^{n,\theta}\right).
\end{align*}
\end{itemize}

\paragraph{Numerical computation of $\alpha$.} The numerical approximation of the optimal control $\alpha$ follows the same steps as in the case of $\widehat{\alpha}$, the additional complexity being given by the fact that we have an additional Brownian motion $W$, which has to be approximated by using another random walk independent of the first two ones.

\subsection{Numerical experiments}    

\paragraph{Parameters estimation} First of all, in order to estimate the dynamics of $\widehat Q$ and $Q$, we use the half-hour electricity data of 300 homes with rooftop solar\footnote{https://www.ausgrid.com.au/Industry/Our-Research/Data-to-share/Solar-home-electricity-data} recorded in Australia from 2012 to 2013. In our simulations, we consider  a typical weekday in July and make the assumption that every weekday in July 2012 has the same seasonality. Therefore, the seasonality is estimated by computing the mean of $\widehat Q$ ($\widehat Q$ being the average of the 300 homes consumption) for every weekday in July. The volatilities of the common and idiosyncratic noises are estimated from the differences between the estimated seasonality and the realizations of all residential households every weekday in July. As a result, we obtain $\sigma^{st} = \sigma^{0} = 0.31$ and $\sigma = 0.56$. The realised $\widehat Q$ (resp. $Q$) and the simulated $\widehat Q$ (resp. $Q$) are shown in Figure \ref{fig:hat_Q} (resp. Figure \ref{fig:Q}).

\begin{figure}[!ht]
    \center
    \subfigure{
        \includegraphics[scale=0.5]{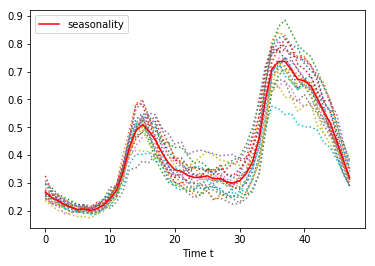}
    }
    \subfigure{
        \includegraphics[scale=0.5]{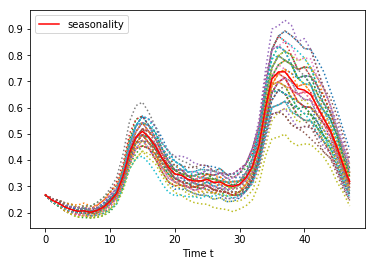}
    }
    \caption{Trajectories of $\widehat Q$ (in kW) with estimated seasonality over 48 half-hours in a weekday in July.}
    \label{fig:hat_Q}
\end{figure}

\begin{figure}[!ht]
    \center
    \subfigure{
        \includegraphics[scale=0.5]{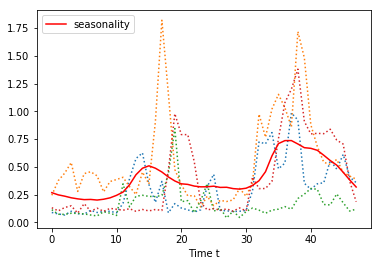}
    }
    \subfigure{
        \includegraphics[scale=0.5]{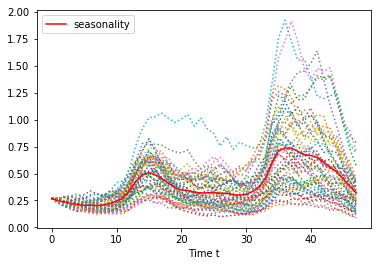}
    }
    \caption{Trajectories of $Q$ (in kW) with estimated seasonality over 48 half-hours in a weekday in July.}
    \label{fig:Q}
\end{figure}
We recognize on the estimated seasonality the typical pattern of households consumption: a peak of power consumption in the morning and a second peak in the evening when people go back home. It can be easily observed that the simulated trajectories look very similar to the observations for $\widehat Q$. The simulated trajectories of $Q$ look less satisfactory than for $\widehat Q$ but we leave for further research the definition of a better model as our focus here is to provide a complete mathematical and numerical treatment of the linear-quadratic case.

To estimate the parameters corresponding to the price function $p$, we use the historical data of global consumption  and spot prices in France \footnote{source: French TSO https://www.services-rte.com} and obtain that $p_0 = 6.16 $ \euro \unskip~ and $p_1 = 0.65$ \euro/GWh (see Figure \ref{fig:price_demand}).

\begin{figure}[!ht]
    \center
    \includegraphics[scale=0.5]{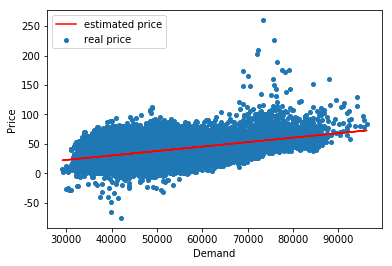}
    \caption{Price (in \euro /MWh) - demand curve (in MW).}
    \label{fig:price_demand}
\end{figure}

 We consider a time horizon $T$ of two days (i.e. 96 half-hours or 48 hours). The parameters are set to $A=150$, $C=80$, $K=50$, $h_0=h_1=0$ and $h_2=100$, which means that each consumer targets a null value for $S$ at the end of the period. The interruptible load contract is set at $\widehat \alpha = 0.1$ kW, a delay \textcolor{black}{$\theta=3 \text{ hours}$} and $\lambda = 2$, $f_0=0$ and $f_1 = 10000$. In particular, the value $\lambda^0 =2$ implies an average number of jumps of 4 over the horizon $T=2$ days \textcolor{black}{(i.e. 48 hours)}. We consider the population to be equally shared ($\pi=0.5$) between \textit{standard consumers} and \textit{consumers with DSM contracts}, the  standard consumers being assumed to have the same consumption dynamics as those with DSM contracts.

\paragraph{Numerical illustrations}
We present the results corresponding to one set of trajectories, which are represented in Figure \ref{fig:trajectoryQJ}. 
\begin{figure}[!ht]
    \center
    \subfigure{
        \includegraphics[scale=0.5]{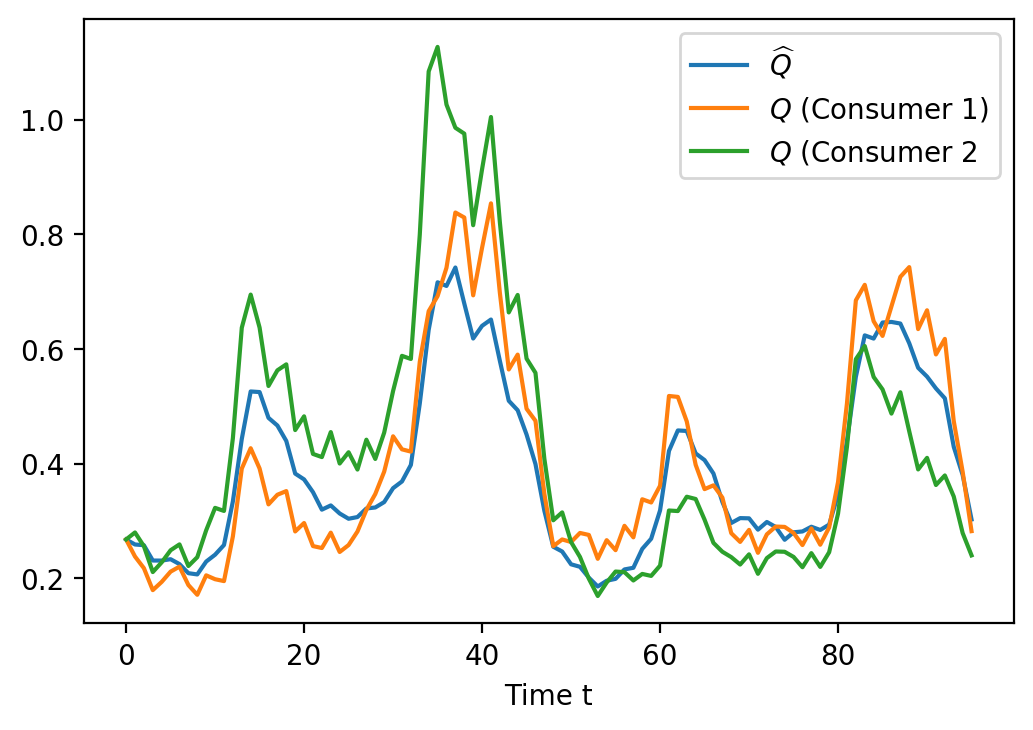}
    }
    \subfigure{
        \includegraphics[scale=0.5]{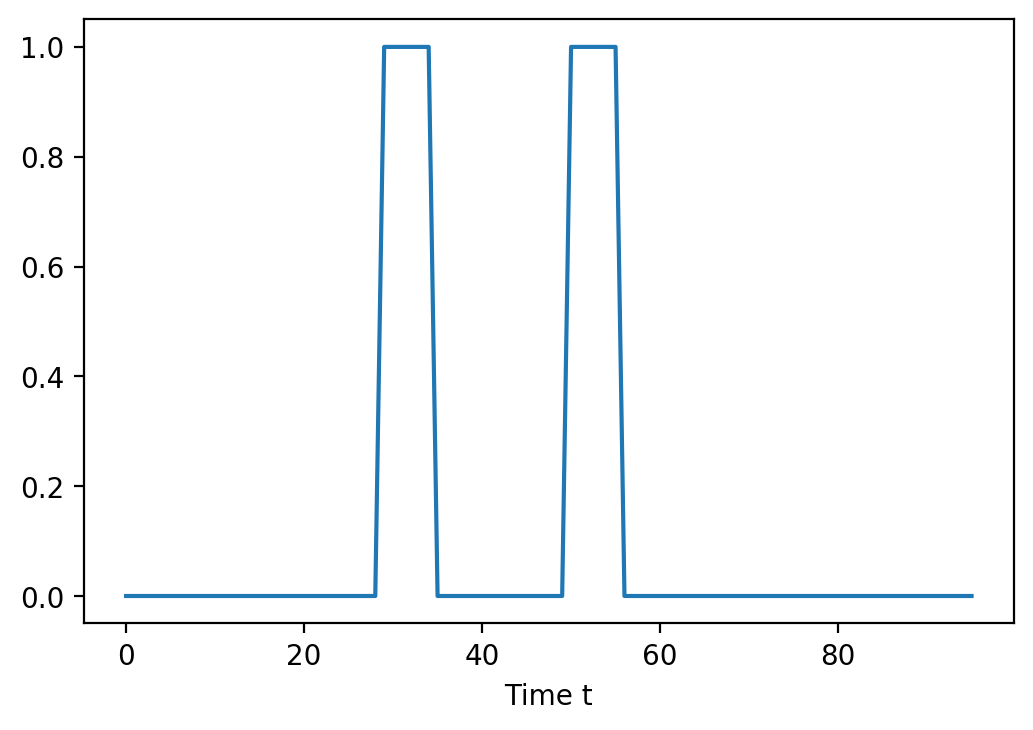}
    }
    \caption{One trajectory of $\widehat Q$  and $Q$ (in kW) for two different consumers (left) and one trajectory of $J$ (right) along time (in half-hours).}
    \label{fig:trajectoryQJ}
\end{figure}
On this set of trajectories, the consumption is globally more important during the first day compared with the second day. Consumer 1 is globally less consuming than consumer 2 and also less than the average of all consumers during the first day, whereas this effect is reversed during the second day. Consumers are faced with two jumps related to the DSM contract during the 48-hour time horizon. 

\paragraph{Numerical results for Real Time Tariff and no Demand-Side Management.}
We first present results when consumers only have RTP tariff and no divergence (i.e. $f_1 = 0$). Figure  \ref{fig:trajectoryQ_alpha_P1only} shows the resulting consumption $Q+\alpha$ for both situations MFG and MFC.  First, we notice that in the MFC the effort $\alpha$ which are made by consumers with DSM contract are more important:  they reduce more their consumption during high demand moment compared to the MFG situation. Indeed, in the MFC case, DSM consumers' action benefits to themselves but also to the standard consumers. Therefore, when they reduce their consumption, they actually reduce power prices which is also beneficial to standard consumers, leading them to make more effort. We also observe the typical \textit{valley-filling phenomenon} we expected: consumers transfer some part of their consumption from high-consumption periods corresponding to high-price periods to low-consumption periods. By doing so, they reduce the cost they have to pay for their consumption. Let us note that in the MFG case, the consumers increase their consumption (by having a positive $\alpha$) during the first hours when their consumption is naturally low but that they also increase their consumption at the first morning peak which may seem less natural. But this can be explained as the initial value of $S$ at time $0$ is $-0.5$ for this simulation. This typically represents the situation when consumers start the period with some postponed consumption they are trying to catch up. This effect is less visible in the MFC situation.

\begin{figure}[!ht]
    \center
    \subfigure{
        \includegraphics[scale=0.5]{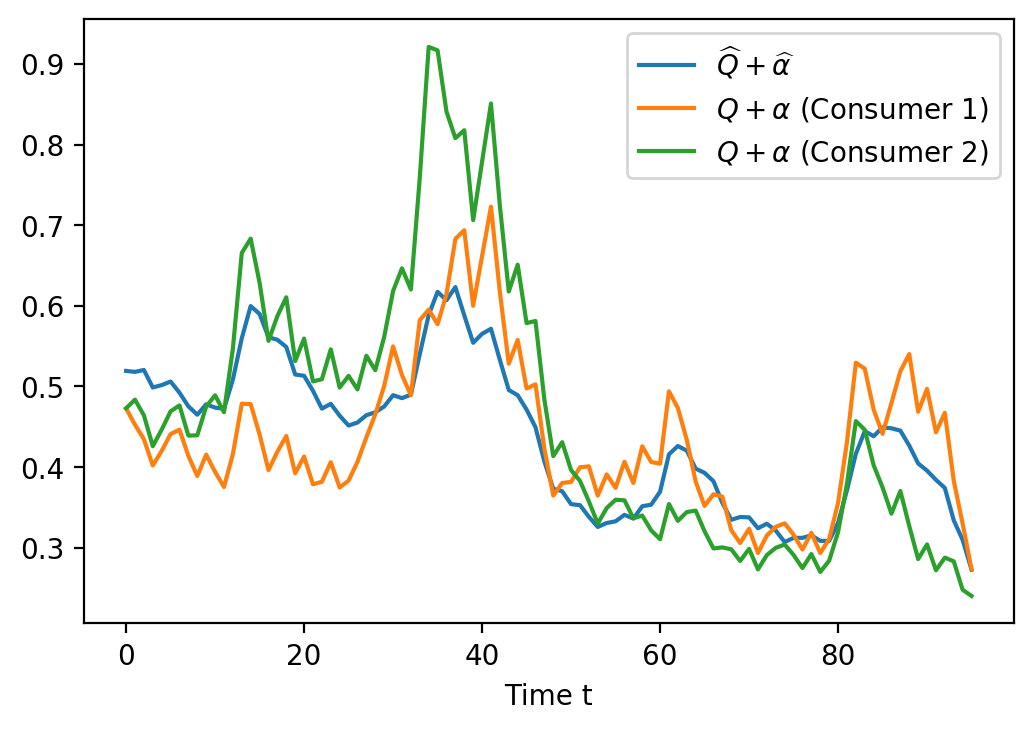}
    }
    \subfigure{
        \includegraphics[scale=0.5]{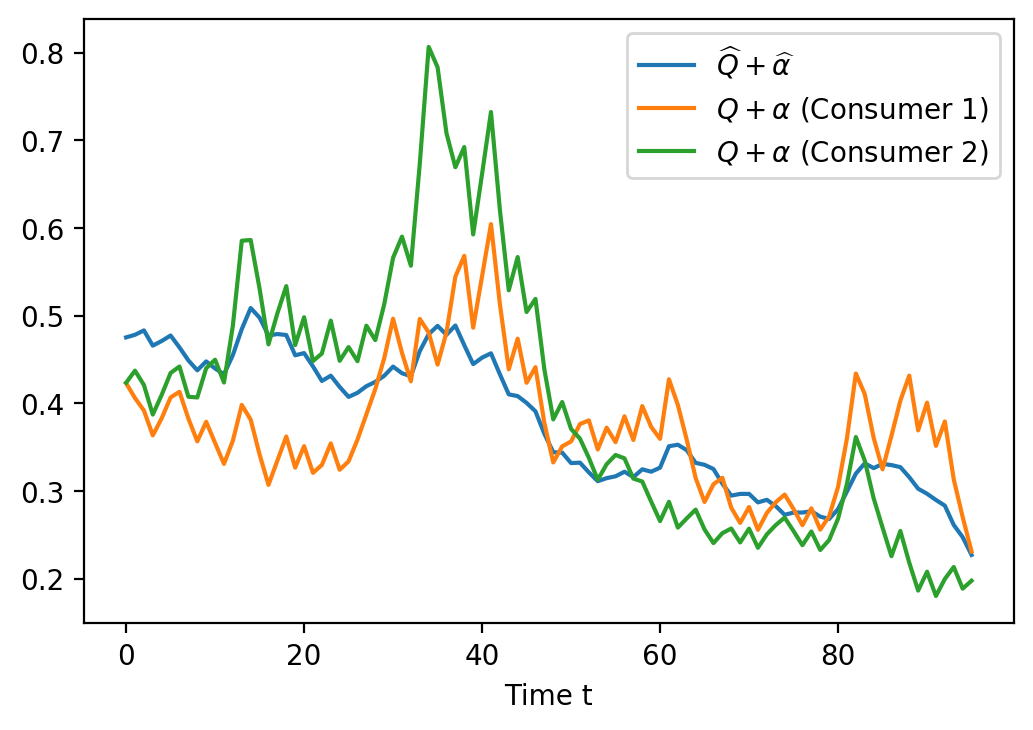} 
    }
    \caption{Trajectories of $\widehat Q + \widehat \alpha$  and $Q+\alpha$ (in kW) for two different consumers for MFG (left) and corresponding trajectories for MFC (right) when $f_1=0$ along time (in half-hours).}
    \label{fig:trajectoryQ_alpha_P1only}
\end{figure}

The consumers' actions have a direct impact on spot prices $p$ represented in Figure \ref{fig:trajectory_price_P1only}. We plot the impact of different proportions $\pi$ of standard consumers in the system. It is clear that the lower $\pi$, the more active consumers in the system are, which implies that their impact is collectively more important. The valley-filling of consumption has a smoothing effect on prices: the peak prices are attenuated, whereas the low prices increase. This effect is more important for MFC for the reasons already given. 
\begin{figure}[!ht]
    \center
    \subfigure{
        \includegraphics[scale=0.5]{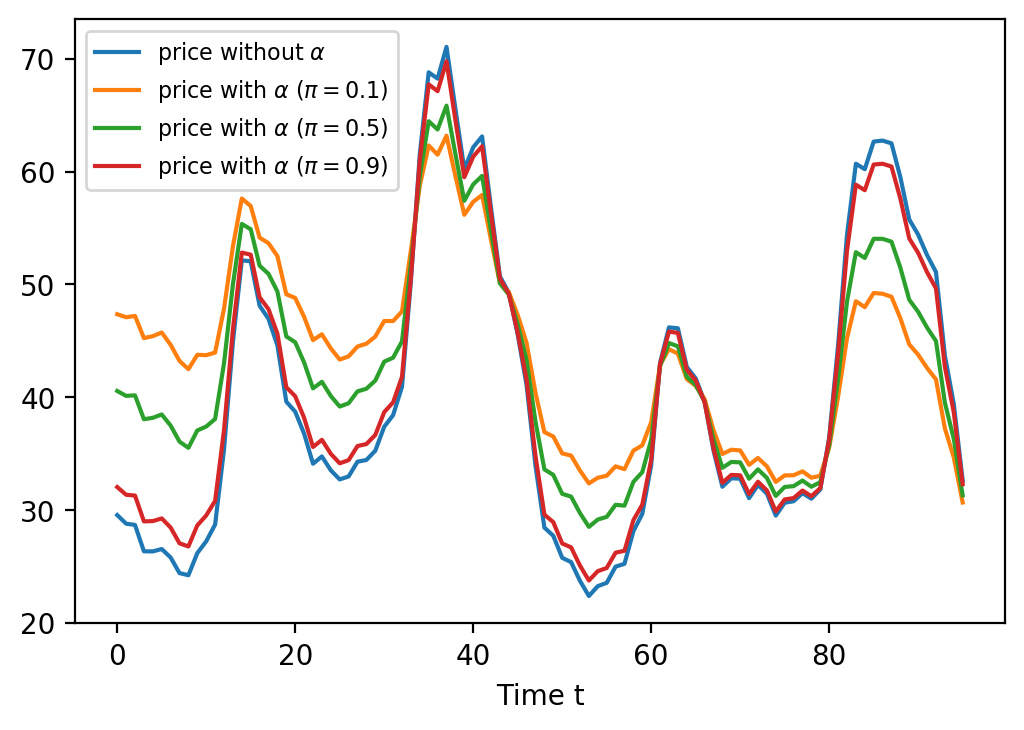}
    }
    \subfigure{
        \includegraphics[scale=0.5]{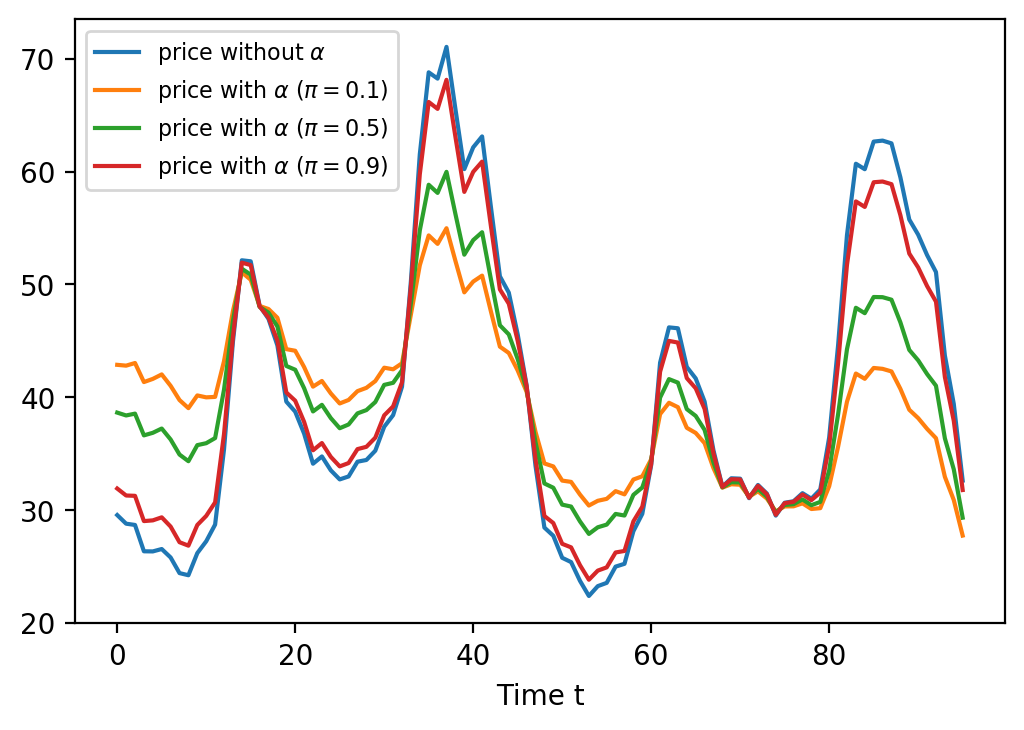} 
    }
    \caption{Trajectories of price $p$ for three different proportions of active consumers for MFG (left) and corresponding trajectories for MFC (right) when $f_1=0$ along time (in half-hours).}
    \label{fig:trajectory_price_P1only}
\end{figure}

Another optimization problem, called $\text{MFC}^{agg}$, could be formulated to represent \emph{the point of view of the aggregator} who coordinates all the DSM consumers, without being interested in the rest of the population:
\begin{align}
V^{C^{agg}}= &\inf_{\alpha \in \mathcal{A}} \mathbb{E} \left[\int_{0}^{T} \left( g(\alpha_t, S_t, Q_t)  + (Q_t+\alpha_t )p\left(\pi Q^{st}+(1-\pi)\widehat{Q}_t + \widehat{\alpha}_t \right)  \right.\right. \nonumber\\ 
& \hspace{1cm} \left.\left. + l(Q_t+\alpha_t)+ J_t^\theta (\tilde Q_t + \alpha_t - \bar \alpha) f \left(\widehat{\tilde Q_t} + \widehat \alpha_t   - \bar \alpha \right) \right) dt +h(S_T)  \right],
\label{eq:MFC_aggregator_objective}
\end{align}
with the function $p$ of class $C_b^1$ and $g,f,l,h$ satisfying Assumption \ref{assump_2.1}.

By taking the point of view of the aggregator of DSM consumers (the associated value of the optimization problem being given by $V^{C^{agg}}$), the equilibrium can also be characterized in the same way as it has been done for the social planner $V^{C}$ in previous section. Furthermore, by comparing the coupling conditions associated to the MFG problem and to $\text{MFC}^{agg}$, it follows that if $\alpha^{\star}_{MFC^{agg}}$ is a mean field optimal control for the problem with pricing rules $p_{MFC^{agg}}$ and $f_{MFC^{agg}}$ then $\alpha^{\star}_{MFC^{agg}}$ is a MFG Nash equilibrium for the MFG  with pricing rules: 
\begin{align}\label{agg}
    p^{MFG}:[0,T] \times {L}^2 \to {L}^2 \nonumber\\
    p_{t}^{MFG}(Q)=\Tilde{p}(t,\pi Q_t^{st}+(1-\pi)Q),
\end{align}
with $\Tilde{p}$ given by:
    $$\Tilde{p}(t,Q)(\omega)=p_{MFC^{agg}}(Q(\omega)) + \left(Q(\omega)-\pi Q_t^{st}(\omega)\right)p_{MFC^{agg}}^{'}(Q(\omega)).$$
In the linear quadratic setting, choosing $p_{MFC^{agg}}(x)=p_0+p_1x$, we obtain $p_t^{MFG}(Q)=p_0+2p_1(1-\pi)Q+p_1\pi Q_t^{st}$ and the results obtained in this section remain valid for this price function.

Centralized decisions might be difficult to implement in practice as they require a large amount of information to be sent from the aggregator to the agents and the other way round. This relation between  the aggregator's problem and an MFG problem enables to implement the decision of the aggregator through a decentralized setting by letting the agents play the MFG with the modified pricing rules given by \eqref{agg}.

The point of view of the aggregator $\text{MFC}^{agg}$ (\ref{eq:MFC_aggregator_objective}) produces intermediary efforts between the MFG and the MFC. The efforts produced are illustrated in figure  \ref{fig:trajectoryQ_alpha_P1only_Aggregator}. This result is expected as the MFG focuses its objective on one single consumers, the MFC on the total consumers, i.e. the standard and DSM consumers whereas the $\text{MFC}^{agg}$ is in between as it is interested in all DSM consumers but not in standard ones.

\begin{figure}[!ht]
    \center
    \includegraphics[scale=0.5]{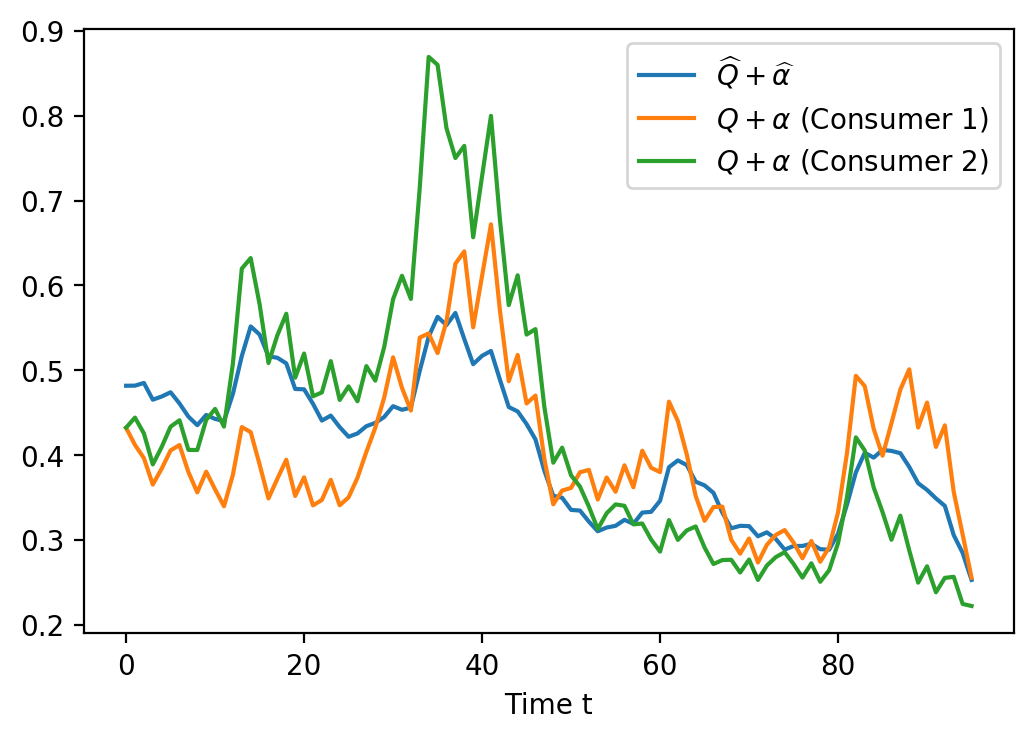}
    \caption{Trajectories of $\widehat Q + \widehat \alpha$  and $Q+\alpha$ (in kW) for two different consumers from Aggregator point of view along time (in half-hours).}
    \label{fig:trajectoryQ_alpha_P1only_Aggregator}
\end{figure}

\paragraph{Numerical results for Demand-Side Management but no Real Time Tariff.}
We now illustrate the situation in which the active consumers only have a DSM incentives but no RTP (i.e. $p_1=0$). In Figure \ref{fig:trajectory_tilde_Q}, we observe that in the MFG case the response of all active consumers ($\widehat{\tilde Q} + \widehat \alpha$) is exactly at the expected DSM level $\bar \alpha$. Each single consumer is not exactly at the target because of their individual constraints and personal situation, but as a whole they manage to satisfy the contract. The results in the  MFC case are not presented as they are quite similar to those obtained for the MFG setting. Indeed, both optimisation problems are very similar since the price $p$ is null. 

\begin{figure}[!ht]
    \center
    
    \includegraphics[scale=0.6]{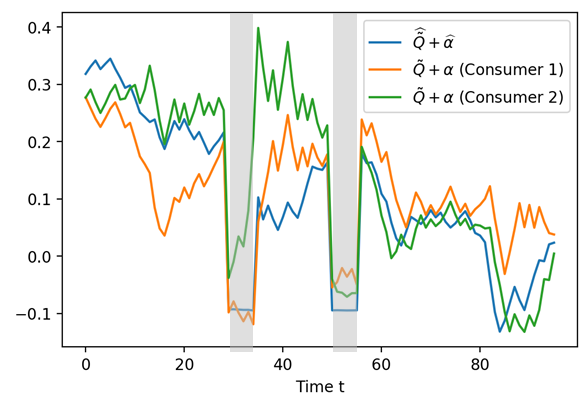}
    
    \caption{Trajectories of $\widehat{ \tilde Q}+ \widehat \alpha$ and $\tilde Q + \alpha$ in the MFG case along time (in half-hours).}
    \label{fig:trajectory_tilde_Q}
\end{figure}

\paragraph{Numerical results for Demand-Side Management and Real Time Tariff.}
When both DSM and RTP are combined, the influence of jumps is still clearly noticeable. In Figure \ref{fig:trajectory_Q_alpha}, we observe that $\widehat{\tilde Q} + \widehat \alpha$ matches $\bar \alpha$ during the jumps episodes for both MFG and MFC cases. During the other periods, i.e. without jumps, consumers react in a very similar manner to what was observed for the simulation with $f_1=0$. The resulting spot price is presented in Figure \ref{fig:trajectory_price}, in particular the price drops during the jumps because the global consumption is reduced. 
\begin{figure}[!ht]
    \center
    \subfigure{
        \includegraphics[scale=0.6]{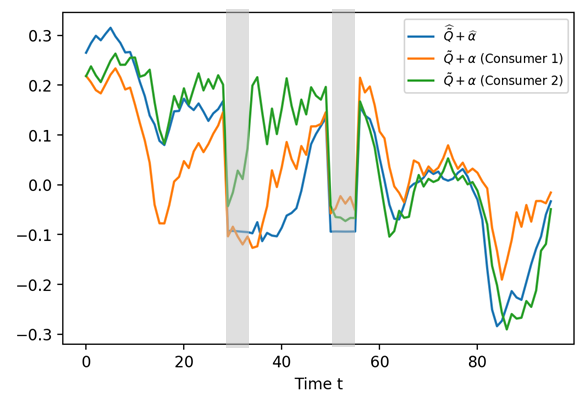}
    }
    \subfigure{
        \includegraphics[scale=0.6]{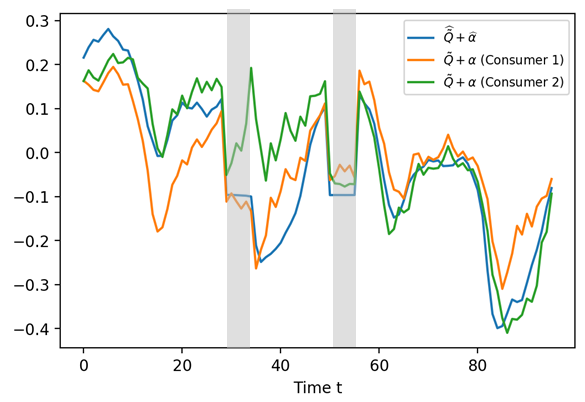}
    }
    \caption{Trajectories of $\widehat{\tilde{Q}} + \widehat \alpha$ (in kW) and $\tilde{Q}+\alpha$ for two different consumers for MFG (left) and for MFC (right) along time (in half-hours).}
    \label{fig:trajectory_Q_alpha}
\end{figure}

\begin{figure}[!ht]
    \center
    
    \includegraphics[scale=0.5]{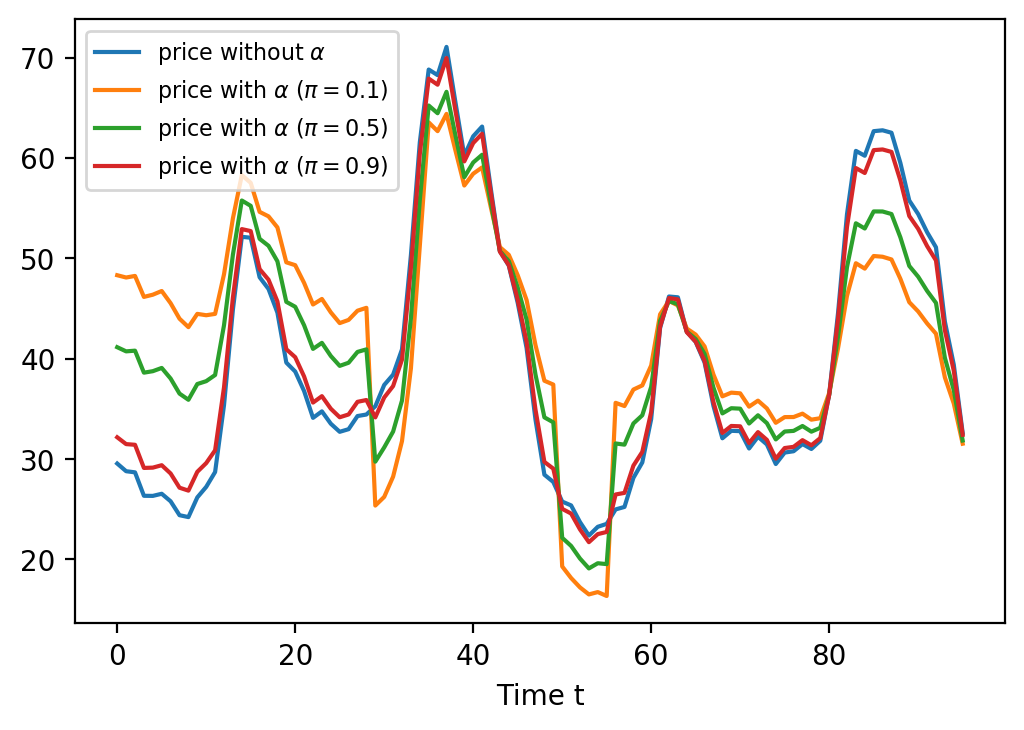}
    
    \caption{Trajectories of price $p$ for different proportion $\pi$ of standard consumers in the system in the MFG setting (jumps episodes are highlighted in grey) along time (in half-hours).}
    \label{fig:trajectory_price}
\end{figure}

\bibliographystyle{siam} %The style you want to use for references.
\bibliography{bibliography}
\end{document}